\documentclass{llncs}
\usepackage{pstricks}
\usepackage{pst-text}
\usepackage{mathrsfs}
\usepackage{graphicx}
\usepackage{amsmath}
\usepackage{amsfonts}
\usepackage{amssymb}
\usepackage{epic}

\newtheorem{observation}{Observation}

\def\qed{\hfill$^{\fbox{\footnotesize{}}}$}

\begin{document}

\title{To Prove Four Color Theorem}

\author{Weiya Yue\inst{1}, Weiwei Cao\inst{2}}
\institute{$^1$ Google Inc., Mountain View, CA 94043\\
$^2$ Graduate School
of Chinese Academy of Sciences, Beijing China 100049\\
\email{weiyayue@hotmail.com}}

\maketitle

\vspace{-5mm}
\section{Abstract}
In this paper, we show that a planar graph $G$ has a color assignment
using $\leq 4$ colors. I.e., every planar graph is $4$-colorable.


\section{Introduction}
It is known that four color theorem is one special case of Hadwiger
conjecture~\cite{HH1943} when $k=5$. I.e., if a graph has its
chromatic number $5$, then there is one $K_5$ minor in it. And the
case when $k=4$ has been proved, i.e. a chromatic number $4$ graph
has one $K_4$ minor~\cite{Dirac1952}.

The four color theorem has been proved assisted by computer for the
first time in 1976 by Kenneth Appel and Wolfgang Haken. A simpler
proof using the same idea and also relied on computer was given in
1997 by Robertson, Sanders, Seymour, and Thomas. Additionally in
2005, the theorem was proven by Georges Gonthier with general
purpose theorem proving software which is also relied on computer.
All these proofs have one thing in common that they are all
complicated computer-assisted proofs which render it unreadable and
uncheckable by hand. {\it None of such proofs is a mathematical
proof.}

In this paper, we will prove that a planar graph $G$ has a color
assignment using $\leq 4$ colors in which $G$'s perimeter is
assigned $\leq 3$ colors. Hence we prove that every planar graph is
$4$-colorable. Moreover, we claim that by using results
of~\cite{WeiyaNote1,WeiyaNote2,WeiyaNote3}, this proof can be
generalized to prove Hadwiger Conjecture.

In Section~\ref{sec-Terminology}, necessary terminologies and
definitions are introduced. In Section~\ref{sec-constraints}, some
results are proved prepared for later use in proof of four color
theorem.

\section{Terminology Definition and Preliminary Results}\label{sec-Terminology}

In this section, Conventional graph theory terminology applied. In
Subsection~\ref{sec-perimeter-trace}, {\em Perimeter Trace} of a
planar graph and {\em cluster} are defined and some their properties
are introduced. In Subsection~\ref{sec-color-ex}, {\em color
collections} is defined and analyzed.

\begin{definition}\label{color-assignment}
To graph $G(V,E)$, one color assignment can be treated as a group of
partitions of $V$, in which one partition is an independent set, and
every partition is assigned with one different color.
\end{definition}

In this paper, we often use $cl$ to denote one color assignment and
also use $cl$ to represent colors used in $cl$. And for convenience,
we use integers to represent colors. Then we can say there is one
color assignment or a set of colors $cl=\{1,2,...,l\}, |cl|=l$. In
$cl$, a color used on vertex $v$ is represented by $color_{cl}(v)$,
when there is confusion, also use $color(v)$ directly.

The terminologies below are used in this paper.
Given a graph $G=(V,E)$, a subgraph $G_s=(V_s,E_s)$ of $G$, vertex
$v\in V$, and a set of vertices $W$: $G_s'(V_s',E_s')=G_s\cup v$
means that in $G_s'(V_s',E_s')$, $V_s'=V_s\cup v$ and $E_s'=E_s\cup
\{edges\ from\ v\ to\ V_s\ in\ G\}$. $G_s'(V_s',E_s')=G_s\cup W$
means that in $G_s'(V_s',E_s')$, $V_s'=V_s\cup W$ and $E_s'=E_s\cup
\{edges\ from\ W\ to\ W\cup V_s\ in\ G\}$.
\subsection{Perimeter Trace}\label{sec-perimeter-trace}

\begin{figure}
\setlength{\unitlength}{1mm}
\begin{pspicture}(1,0)(10,3)
\pscircle[linecolor=gray](1,1.5){0.2}
\pscircle[linecolor=gray](1.75,2.2){0.2}
\pscircle[linecolor=gray](2.5,1.8){0.2}
\pscircle[linecolor=gray](3.25,2.2){0.2}
\pscircle[linecolor=gray](4.0,1.5){0.2}
\pscircle[linecolor=gray](3.25,0.8){0.2}
\pscircle[linecolor=gray](2.5,1.2){0.2}
\pscircle[linecolor=gray](1.75,0.8){0.2}
\rput(1,1.5){$v_1$} \rput(1.75,2.2){$v_2$} \rput(2.5,1.8){$v_3$}
\rput(3.25,2.2){$v_4$} \rput(4.0,1.5){$v_5$} \rput(3.25,0.8){$v_6$}
\rput(2.5,1.2){$v_7$} \rput(1.75,0.8){$v_8$}
\psline(1.1,1.6)(1.6,2.05) \psline(1.85,2.1)(2.4,1.9)
\psline(2.5,1.6)(2.5,1.4)
\psline(2.6,1.9)(3.15,2.1)\psline(3.35,2.1)(3.9,1.6)
\psline(3.9,1.4)(3.35,0.9)\psline(3.15,0.9)(2.6,1.1)\psline(2.4,1.1)(1.85,0.9)\psline(1.6,0.9)(1.15,1.35)
\pscircle*[linecolor=gray](1.75,1.5){0.1}\pscircle*[linecolor=gray](3.25,1.5){0.1}
\psline(1.75,1.5)(1.2,1.5)\psline(1.75,1.5)(1.75,2.0)\psline(1.75,1.5)(2.5,1.8)\psline(1.75,1.5)(2.5,1.2)\psline(1.75,1.5)(1.75,1.0)
\psline(3.25,1.5)(2.6,1.8)\psline(3.25,1.5)(3.25,2.0)\psline(3.25,1.5)(3.9,1.5)\psline(3.25,1.5)(3.25,1.0)\psline(3.25,1.5)(2.5,1.2)
\rput(2.5,0){$(a)$}
\pscircle[linecolor=green](1,1.5){0.2}
\pscircle[linecolor=blue](1.75,2.2){0.2}
\pscircle[linecolor=green](2.5,1.8){0.2}
\pscircle[linecolor=blue](3.25,2.2){0.2}
\pscircle[linecolor=green](4.0,1.5){0.2}
\pscircle[linecolor=blue](3.25,0.8){0.2}
\pscircle[linecolor=orange](2.5,1.2){0.2}
\pscircle[linecolor=blue](1.75,0.8){0.2}
\rput(1,1.5){$v_1$} \rput(1.75,2.2){$v_2$} \rput(2.5,1.8){$v_3$}
\rput(3.25,2.2){$v_4$} \rput(4.0,1.5){$v_5$} \rput(3.25,0.8){$v_6$}
\rput(2.5,1.2){$v_7$} \rput(1.75,0.8){$v_8$}
\psline(1.1,1.6)(1.6,2.05) \psline(1.85,2.1)(2.4,1.9)
\psline(2.5,1.6)(2.5,1.4)
\psline(2.6,1.9)(3.15,2.1)\psline(3.35,2.1)(3.9,1.6)
\psline(3.9,1.4)(3.35,0.9)\psline(3.15,0.9)(2.6,1.1)\psline(2.4,1.1)(1.85,0.9)\psline(1.6,0.9)(1.15,1.35)
\pscircle*[fillcolor=red,linecolor=red](1.75,1.5){0.1}\pscircle*[linecolor=red](3.25,1.5){0.1}
\psline(1.75,1.5)(1.2,1.5)\psline(1.75,1.5)(1.75,2.0)\psline(1.75,1.5)(2.5,1.8)\psline(1.75,1.5)(2.5,1.2)\psline(1.75,1.5)(1.75,1.0)
\psline(3.25,1.5)(2.6,1.8)\psline(3.25,1.5)(3.25,2.0)\psline(3.25,1.5)(3.9,1.5)\psline(3.25,1.5)(3.25,1.0)\psline(3.25,1.5)(2.5,1.2)
\rput(2.5,0){$(a)$}
\pscircle[linecolor=gray](5.5,1.5){0.2}
\pscircle[linecolor=gray](6.25,2.2){0.2}
\pscircle[linecolor=gray](7.0,1.5){0.2}
\pscircle[linecolor=gray](7.75,2.2){0.2}
\pscircle[linecolor=gray](8.5,1.5){0.2}
\pscircle[linecolor=gray](7.75,0.8){0.2}
\pscircle[linecolor=gray](6.25,0.8){0.2}
\rput(5.5,1.5){$v_1$} \rput(6.25,2.2){$v_2$} \rput(7.0,1.5){$v_3$}
\rput(7.75,2.2){$v_4$} \rput(8.5,1.5){$v_5$} \rput(7.75,0.8){$v_6$}
\rput(6.25,0.8){$v_7$}
\psline(5.6,1.6)(6.1,2.05) \psline(6.35,2.1)(6.85,1.6)
\psline(7.15,1.6)(7.65,2.1)\psline(7.85,2.1)(8.4,1.6)
\psline(8.4,1.4)(7.85,0.9)\psline(7.65,0.9)(7.15,1.4)
\psline(6.85,1.4)(6.35,0.9)\psline(6.1,0.9)(5.65,1.35)
\pscircle*[linecolor=gray](6.25,1.5){0.1}\pscircle*[linecolor=gray](7.75,1.5){0.1}
\psline(6.25,1.5)(5.7,1.5)\psline(6.25,1.5)(6.25,2.0)\psline(6.25,1.5)(7.0,1.5)\psline(6.25,1.5)(6.25,1.0)
\psline(7.75,1.5)(7.2,1.5)\psline(7.75,1.5)(7.75,2.0)\psline(7.75,1.5)(8.4,1.5)\psline(7.75,1.5)(7.75,1.0)
\rput(7,0){$(b)$}
\pscircle[linecolor=gray](10,1.5){0.2}
\pscircle[linecolor=gray](10.75,2.2){0.2}
\pscircle[linecolor=gray](12.25,2.2){0.2}
\pscircle[linecolor=gray](10.75,1.5){0.2}
\pscircle[linecolor=gray](11.5,1.5){0.2}
\pscircle[linecolor=gray](12.25,1.5){0.2}
\pscircle[linecolor=gray](13.0,1.5){0.2}
\pscircle[linecolor=gray](12.25,0.8){0.2}
\pscircle[linecolor=gray](10.75,0.8){0.2}
\rput(10,1.5){$v_1$} \rput(10.75,2.2){$v_2$}
\rput(10.75,1.5){$v_3$}\rput(11.5,1.5){$v_4$}
\rput(12.25,1.5){$v_5$}
\rput(12.25,2.2){$v_6$}\rput(13.0,1.5){$v_7$}
\rput(12.25,0.8){$v_8$} \rput(10.75,0.8){$v_9$}
\psline(10.1,1.6)(10.6,2.05) \psline(12.35,2.1)(12.9,1.6)
\psline(12.9,1.4)(12.35,0.9) \psline(10.6,0.9)(10.15,1.35)
\psline(10.75,1.65)(10.75,2.0) \psline(10.75,1.35)(10.75,1.0)
\psline(12.25,1.65)(12.25,2.0)\psline(12.25,1.35)(12.25,1.0)
\psline(10.1,1.5)(10.65,1.5) \psline(10.85,1.5)(11.4,1.5)
\psline(11.6,1.5)(12.15,1.5) \psline(12.35,1.5)(12.9,1.5)
\rput(11.5,0){$(c)$}
 \end{pspicture}
\caption{Examples of Perimeter Trace}\label{fig-0}
\end{figure}
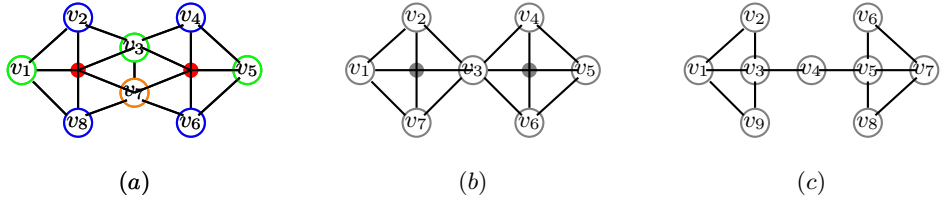

\begin{definition}\label{perimeter-trace}
In a connected graph $G(V,E)$, define its perimeter trace as a
walking on a series of vertices $\{v_1,v_2,...,v_x,v_1\}$, in which
if there are $v_i,v_j,v_k,v_l, i<j<k<l$ then every path between
$v_i,v_k$ intersects with every path between $v_j,v_l$ by assuming
there is edge for every pair of $v_y,v_{(y+1)mod\ (x+1)}$.
\end{definition}

The beginning and ending of a perimeter trace are considered the
same vertex. For convenience, we say two vertices appear
continuously if in a perimeter trace one follow another one without
separation by other vertices. In a planar graph, if we trace its
perimeter, we can get a perimeter trace. For example, in
Figure~\ref{fig-0}.$a$, $\{v_1,v_2,v_3,v_4,v_5,v_6,v_7,v_8,v_1\}$ is
a perimeter trace. When the planar graph is $1$-connected, as shown
in Figure~\ref{fig-0}.$b$, $\{v_1,v_2,v_3,v_4,v_5,v_6,v_3,v_7,v_1\}$
is a perimeter trace where $v_3$ appears twice. In
Figure~\ref{fig-0}.$c$,
$\{v_1,v_2,v_3,v_4,v_5,v_6,v_7,v_8,v_5,v_4,v_3,v_9,v_1\}$ is a
perimeter trace. It is worthy to notice that in a perimeter trace,
one vertex may appear more than one time. 


\begin{observation}\label{perimeter-trace-more-than-1-time}
In a perimeter trace $s$ of graph $G(V,E)$, if one vertex appears
more than $1$ time noncontinuously, the vertex is a cut vertex.
\end{observation}

A subset of a perimeter trace $S$ is called a
$cluster$ of $S$.
For example, in Figure~\ref{fig-0}.$a$,
$\{v_1,v_2,v_3,v_4,v_5,v_6,v_7,v_8,v_1\}$ is a perimeter trace.
Because one cluster can be chose roundly, $``v_7,v_8,v_1"$ is one
cluster. $\{``v_1,v_2,v_3", ``v_4,v_5",``v_6,v_7,v_8"\}$ and
$\{``v_1,v_2,v_3", ``v_3", ``v_3",``v_3",
``v_4,v_5",``v_6,v_7,v_8"\}$ are its two sets of clusters.

\begin{proposition}\label{planar-perimeter}
If connected graph $G(V,E)$ is planar, then a walking generates a
perimeter trace if and only if the walking is on a perimeter. If $S$
is $G$'s perimeter trace and $u\in S$ which is not a cut vertex,
after deleting $u$ from $G$, $\{S\setminus u\}\cup \{N(u)\setminus
S\}$ is a perimeter trace.
\end{proposition}
\begin{proof}
This follows from known properties of planar graph. After deleting
$u\in S$ from $G$, the new perimeter trace is the same except $u$ is
replace by $N(u)\setminus S$.

\qed
\end{proof}

\subsection{Series of Clusters}\label{series-of-clusters}

Given a cluster on a perimeter trace, we can create a cycle by connecting cluster vertices
according to the order of perimeter trace, and call the cycle as cluster cycle.

\begin{definition}\label{independent-clusters}
Given two clusters $cs_1,cs_2$, if their cluster cycles have no edge overlapped with each other,
then we call these two clusters independent from each other.
\end{definition}

\begin{definition}\label{series-clusters}
On a perimeter trace, we can define a set of clusters, if all of them are independent from each other,
then can all this a series of clusters.
\end{definition}
In this paper, we only discuss independent clusters.

\begin{lemma}\label{how-series-clusters-overlapped}
Given a perimeter trace, and a series of cluster on the trace, the perimeter trace and all cluster cycles
of the series of clusters consist an outerplanar graph.
\end{lemma}
\begin{proof}
This follows from definition of outerplanar graph.
\qed
\end{proof}
In~\cite{WeiyaNote3}, there are more discussions about outerplanar graph, where a geometric view of
outerplanar graph is given.


\subsection{Color Collections}\label{sec-color-ex}

 Given $m\geq 0$ colors, $n\leq m$, there
are $C_{m}^{n}$ combinations of $n$ colors. Assume we are using
colors $M=\{1,2,...,m\}$.
\begin{definition}\label{color-collection}
If there is one set of colors $L\subseteq M$ with $|L|=l\leq n$, we
say $L$ can be extended to be $n$ colors by adding colors from
$M\setminus L$, and we call those collections as $n$-collection of
$L$ respect to $M$.
\end{definition}
Easy to see, there are $C_{m-l}^{n-l}$ different $n$-collections of
$L$ respect to $M$. If $l=0$, then such collections are called $M$'s
$n$-collections. We say a cluster is colored by a $n$-collection, if
the colors used on the cluster is a subset of a $n$-collection of
$M$. Here we use $cn(\Upsilon_i)$ to represent the $n$-collection
coloring of cluster $\Upsilon_i$.

In this paper below, we will assume $M=\{1,2,3,4\}$ where $m=4$, and
$n=3$ color collections are used. Without confusion, when we say
color collection, it means $3$-collection respect to $M$. Hence
there are $C_4^3=4$ $3$-collections such that
$\{1,2,3\},\{1,2,4\},\{1,3,4\},\{2,3,4\}$.

If $l=2$ with color-set $c=\{1,2\}$, then $c$ can be extended in
$C_{4-2}^{3-2}=2$ ways to be $\{1,2,3\},\{1,2,4\}$. {\it Below if we
are talking about $n-collection$, a color set whose cardinality is
$\leq n$ will be considered as equivalent with all its collections.}
I.e., if we are talking about $3-collection$, color set $\{1,2\}$ is
equivalent with $\{1,2,3\},\{1,2,4\}$. So if we are talking about
the cardinality of a set of $n-collection$, all color-sets are
extended to be $n-collections$. I.e., when talking $3-collection$,
$\{1,2\}$ has its cardinality as $2$.

For example, in Figure~\ref{fig-0}.$a$ we use colors
$M=\{1=red,2=green,3=blue,4=orange\}$ and $3$-collections to color
clusters $\{v_8,v_1,v_2\}$, $\{v_4,v_5,v_6\},\{v_7\}$ on perimeter
trace $\{v_1,v_2,v_3,v_4,v_5,v_6,v_7,v_8,v_1\}$.
$cn(\{v_8,v_1,v_2\})=\{2,3\}=\{\{2,3,1\},\{2,3,4\}$, means cluster
$\{v_8,v_1,v_2\}$ can be considered being colored by $3$-collection
$\{2,3,1\}$ or $\{2,3,4\}\}$. Also there are
$cn(\{v_4,v_5,v_6\})=\{2,3\}=\{\{2,3,1\},\{2,3,4\}\}$, and
$cn(\{v_7\})=\{4\}=\{\{1,2,4\},\{1,3,4\},\{2,3,4\}\}$.

\begin{definition}\label{collection-consistent-or-not}
Given two sets of collections $C_1,C_2$, if $C_1\setminus
C_2\not=\emptyset$ and $C_2\setminus C_1\not=\emptyset$, then
$C_1,C_2$ are inconsistent; otherwise, they are consistent.
\end{definition}

\section{Solution Space and Its Kernel}\label{solution-space}

\begin{definition}\label{cluster-solution-space}
If we treat each color assignment for a cluster $cs$ as a solution,
then all solutions of the cluster together consist the solution space $ss$ of it.
\end{definition}

\begin{definition}\label{kernel-cluster-solution-space}
Given a cluster, we can add edges to it to get an outerplanar graph, and all color assignments
of this graph is called the kernel solution space $ks$ to the cluster.
\end{definition}

\textit{Given a solution space, name its subset where vertices $W\subseteq cs$ are colored with specific colors
as solution space according $W$}. Next we define the solution space of a series of clusters.

\begin{definition}\label{division-solution-space}
$\{cs_1, cs_2,\ldots, cs_k\}$ are a series of clusters, there
is a solution space for $cs_1$. And for an arbitrary solution of $cs_1$,
there is a solution space for $cs_2$. Iteratively, for each solution of $cs_i$,
there is a solution space of $cs_{i+1}$. This tree structure solution space is
called space of the series of clusters. If every solution space in the solution
tree is a kernel solution, then we call this solution space is complete.
\end{definition}

Easy to see, every node of this tree structure solution
 space is a solution space for a cluster.
For example, if we have two clusters $cs_1, cs_2$. There
 is $\{1,2,3\}, \{1,2,4\}$ for $cs_1$.
And if $cs_1$ is assigned with $\{1,2,3\}$, then $cs_2$
 can be assigned $\{1,2,3\}, \{1,3,4\}$;
Similarly, if $cs_1$ is assigned with $\{1,2,4\}$, then
 $cs_2$ can be assigned with $\{1,2,3\}, \{2,3,4\}$. All these 4 solutions together form the solution space of clusters $cs_1, cs_2$.

From the definition, any cluster solution for $cs_i$ has a
corresponding solution space for $cs_{i+1}$. We have special interests in cluster solution using only certain number of colors, so for convenience, we use $ss_i$ to represent the solution space using $\leq i$ colors.
In this paper, without explicit statement, we are talking about $i=4$ colors.

\begin{lemma}\ref{always-3-colors}
Every kernel solution space has a color assignment using $\leq 3$ colors. In a complete solution space,
each solution node has a color assignment using $\leq 3$ colors. Specifically, if $x$ colors are
\end{lemma}
\begin{proof}
As outerplanar graph can be colored with $\leq 3$ colors, by Definition~\ref{kernel-cluster-solution-space},
every kernel solution space has a solution using $\leq 3$ colors.
By Definition~\ref{series-clusters}, one cluster can have at most "abab"
\qed
\end{proof}

\begin{lemma}\label{kernel-solution-merge}
If $cs=cs_1\bigcup cs_2$, then two kernel solution sets $ks_1, ks_2$ of $cs_1,cs_2$ respectively
can be merged together to be a kernel solution set for $cs$.
\end{lemma}
\begin{proof}
Assume $G_1, G_2$ are the two outerplanar graph corresponding to $ks_1, ks_2$.
The outerplanar graph on $cs$ can be set to be $G=G_1\bigcup G_2$.
Assume after merging, there is a solution $cl$ missed from a kernel solution set for $cs$.
But the partial solution on $cl(G_1)$ and $cl(G_2)$ should both exist in $ks_1, ks_2$ respectively,
which can be merged to be $cl$ for $G$, which is a contradiction.
\qed
\end{proof}

\begin{lemma}\label{kernel-solution-switch}
Given a cluster $cs=\{v_1, v_2,\ldots,v_n\}$ and its kernel solution space $ks$, assume $G=(V,E)$ is $ks$'s corresponding outerplanar graph
and there is no edge $e(v_1, v_n)$. Then in this kernel solution space:
\begin{enumerate}
\item there is solution where $cl(v_1)\not = cl(v_n)$.
\item subset of $ks$ where $cl(v_1)\not = cl(v_n)$ is not empty, and is the kernel solution space for $cs$ whose corresponding outerplanar
      graph is $G'=G\bigcup e(v_1,v_n)$.
\end{enumerate}
\end{lemma}
\begin{proof}
In $G$, as there is no edge $e(v_1, v_n)$, there is a cut-vertex $v\in V$. So for a solution to $G$, if $cl(v_1)=cl(v_n)$,
by doing color projection, we can set $cl(v_1)\not = cl(v_n)$, and easy to see the new color assignment is still a solution
to $G$. By definition of kernel solution space, we can conclude the new color assignment belongs to $ks$.
Then easy to see the subset $ks_{cl(v_1)\not=cl(v_n)}$ of $ks$ where $cl(v_1)\not = cl(v_n)$ is not empty.
$ks_{cl(v_1)\not=cl(v_n)}$ is a subset of solution space of $G'$. Assume there is a solution to $G'$ but not in
$ks_{cl(v_1)\not=cl(v_n)}$. But as $G=G'\setminus e(v_1,v_n)$, this solution $\in ks$. Hence it is also
$\in ks_{cl(v_1)\not=cl(v_n)}$ which is a contradiction.
\qed
\end{proof}

\begin{lemma}\label{outerplanar-complete-solution-set}
In an outerplanar graph, we can define a series of clusters on its outer face, and
those clusters have a complete solution space.
\end{lemma}
\begin{proof}
Assume $G=\{v_1, v_2,\ldots, v_n\}$. If there is no edge $e(v_1,v_n)$, the proof is simpler.
So here we only prove the case where there is edge $e(v_1,v_n)$. Set $G'=G\setminus e(v_1,v_n)$.
Then there is cut vertex $v$ in $G'$, and after deleting $v$ from $G$, there are two graphs $G'_1$ and $G'_2$.
By induction, $G'_1$ and $G'_2$ both have complete solution set. And if a cluster is divided into $G'_1$ and $G'_2$, then by Lemma~\ref{kernel-solution-merge}, they can be merged together to be a complete solution set for this cluster.
By Lemma~\ref{kernel-solution-switch}, we can get the complete solution set to the graph $G'\bigcup e(v_1,v_n)$,
which is $G$.
\qed
\end{proof}

\begin{definition}\label{vertex-mapping}
Assume two graphs $G_1=(V_1,E_1)$ and $G_2=(V_2,E_2)$, and without losing generality $|V_2| \leq |V_1|$,
then we can define a $1:1$ mapping on $V_2$ and subset $V_{G_2}\subseteq V_1$ where $|V_2|=|V_{G_2}$.
If there is a color assignment $cl$ to $G_1$, and $cl(V_{G_2}$ is a color assignment for $G_2$,
then we say $G_1$ and $G_2$ have overlapping on their solutions. I.e., there is overlapping on
solution spaces of $G_1$ and $G_2$.
\end{definition}
Similarly, we can define the overlapping of solution space for clusters as whose solution spaces correspond
to solutions of outerplanar graphs.

\begin{theorem}\label{outerplanar-graph-solution-space-overlapping}
An outerplanar graph has its solution overlapping with another outerplanar graph for an arbitrary $1:1$ mapping.
\end{theorem}
\begin{proof}
Assume there are two graphs $G_1=(V_1,E_1)$ and $G_2=(V_2,E_2)$,
and $|V_2| \leq |V_1|$. For convenience, assume $|V_1|=n_1,|V_2|=n_2$,
$V_1=\{v_1,v_2,\ldots,v_{n_1}\}$, $V_2=\{u_1,u_2,\ldots,v_{n_2}\}$.
Prove this by doing induction on the cardinality of the outerplanar graph $|G_1|+|G_2|$.
When $n_1=n_2=1$, this is obvious.

When $n_1+n_2>2$ and an arbitrary mapping:
\begin{enumerate}
\item there is no edge $e(u_1, u_{n_2})$. If $n_2\leq 2$, the case $n_2=1$ is trivial.
 And if $n_2=2$, then there are cases:
   \begin{enumerate}
   \item There is no edge $e(v_1,v_{n_1})$. Then there is cut vertex $v_i\in V$. Assume there are
   two graphs $G_{11},G_{12}$ by splitting graph $G_1$ on $v_i$. If the mapping happens between
   $G_{11}$ and $G_2$ or $G_{12}$ and $G_2$, then we can do induction on them. Prove the case where
   the mapping happens between $G_{11}$ and $G_2$. By induction, $G_{11}$ has overlapping with $G_2$ on $cl_{11}$.
   Then $cl$ can be merged with color assignment $cl_{12}$ of $G_{12}$ where $cl_{12}(v_i)=cl_{11}(v_i)$ which
   can always achieved by color projection.
   And the mapping happens on $G_{11},G_{12}$ and $G_2$, i.e., $u_1$ is mapped to $v_{u_1}$ in $G_{11}$ and $u_2$ is
   mapped to $v_{u_2}$ in $G_{12}$. On $G_{11}$, we can define cluster $cs=\{v_{u_1},v_i\}$, and define an outerplanar
   graph $G=\{\{q_1,q_2\}, \{e(q_1,q_2)\}\}$ with mapping $M(v_{u_1})=q_1,M(v_i)=q_2$. By induction, $G_{11}$ and $G$
   has overlapping in their solution spaces. Assume the overlapped color assignment is $cl_{11}$ where $cl_{11}(v_{u_1})\not =cl_{11}(v_i)$.
   So easy to see, by color projection, we can get a color assignment $cl_{12}$ for $G_{12}$ where $cl_{11}(v_{u_1})\not =cl_{12}(v_{u_2})$.
   \item\label{n-2-2} There is edge $e(v_1,v_{n_1})$. Set graph $G'_1=G\setminus e(v_1,v_{n_1})$, then there is cut vertex $v_i$.
   Assume by splitting at $v_i$, can get two graphs $G'_{11},G'_{12}$. Without losing generality, we can assume there are
   edges $e(v_1,v_i)$ and $e(v_i,v_{n_1})$. Similar as above, we can define clusters $cs_1=\{v_{u_1},v_i\}$ and
   $cs_2=\{v_i,v_{u_2}\}$ and get color assignment $cl_1,cl_2$ to graphs $G'_{11},G'_{12}$ respectively satisfying
   $cl_1(v_1)\not=cl_1(v_i),cl_1(v_{u_1})\not=cl_1(v_i)$, $cl_2(v_i)\not=cl_2(v_{n_1}),cl_2(v_i)\not=cl_2(v_{u_2})$.
   By color projection, we can claim $cl_1(v_1)\not=cl_2(v_{n_1})$ and hence such a color assignment for $G'_1$ can be used
   by graph $G_1$. There are four subcases.
     \begin{enumerate}
     \item $cl_1(v_1)=cl_1(v_{u_1}), cl_2(v_{u_2})=cl_2(v_{n_1})$. So $cl_1(v_{u_1})\not=cl_2(v_{u_2})$.
     \item $cl_1(v_1)\not=cl_1(v_{u_1}), cl_2(v_{u_2})=cl_2(v_{n_1})$. By doing color projection and with $4$ colors to be used,
           $cl_1(v_{u_1})$ has two choices, so can always choose the one different from $cl_2(v_{u_2})$.
     \item $cl_1(v_1)=cl_1(v_{u_1}), cl_2(v_{u_2})\not=cl_2(v_{n_1})$. Similar as above case.
     \item $cl_1(v_1)\not=cl_1(v_{u_1}), cl_2(v_{u_2})\not=cl_2(v_{n_1})$. Similar as above case.
     \end{enumerate}
   \end{enumerate}
 If $n_2>2$, then there is a cut vertex $u_i$ for $G_2$. And by cutting on $u_i$, we get two graphs $G_{21}$ and $G_{22}$ whose solution space
 are $SS_{21}$ and $SS_{22}$ respectively.
 $G_1$ has corresponding two clusters to $cs_1$ and $cs_2$ according to the mapping and $G_{21},G_{22}$.
 By Lemma~\ref{outerplanar-complete-solution-set}, $G_1$ has a complete solution set to clusters
 $cs_1,cs_2$. Assume $ks_1$ is the kernel solution space in the complete solution set for $cs_1$,
 and $G_{cs_1}$ is the corresponding outerplanar graphs.
 By induction, $ks_1\bigcap ss_{21}\not=\emptyset$. Assume $cl_1=ks_1\bigcap ss_{21}$. Then by
 definition of complete solution set, there is $ks_2$ as the kernel solution space for $cs_2$ according
 to $ks_1$ in this complete solution space. By induction again, there is $ks_2\bigcap ss_{22}\not=\emptyset$, and
 assume the corresponding color assignment of $G_1$ is $cl_{12}$.
 By Lemma~\ref{kernel-solution-merge}, $ss_{21}$ and $ss_{22}$ can be merged together to be a kernel
 solution space $ss_2$ to $G_2$. Easy to see, $ss_2$ has overlapping with $cl_{12}$.
 Hence $G_1,G_2$ has overlapping in their solution space.

\item If there is edge $e(u_1, u_{n_2})$ in $G_2$, set $G'_2=G_2\setminus e(u_1, u_{n_2})$.
 Here we can assume $n_2>2$, as cases $n_2\leq 2$ has been discussed in above scenario. But we still
 need to discuss in two cases $n_2=3$ and $n_2>3$.

 If $n_2=3$, then $G_2=\{u_1,u_2,u_3\}$ is a simple cycle graph.
 \begin{enumerate}
   \item There is no edge $e(v_1,v_{n_1})$. This can be proved similar as case below, and simpler. So we
   skip its proof here.
   \item There is edge $e(v_1,v_{n_1})$. Set graph $G'_1=G\setminus e(v_1,v_{n_1})$.
   Then there is cut vertex $v_i\in V$ in $G'_1$. Assume there are
   two graphs $G'_{11},G'_{12}$ by splitting graph $G'_1$ on $v_i$. If the mapping happens between
   $G'_{11}$ and $G_2$ or $G'_{12}$ and $G_2$, then we can do induction on them. Prove the case where
   the mapping happens between $G'_{11}$ and $G_2$. By induction, $G'_{11}$ has overlapping with $G_2$ on $cl_{11}$.
   Then $cl$ can be merged with color assignment $cl_{12}$ of $G_{12}$ where $cl_{12}(v_i)=cl_{11}(v_i)$ which
   can always achieved by color projection.
   If the mapping happens on $G'_{11},G'_{12}$ and $G_2$, we can assume $u_1$ is mapped to $G'_{11}$ on vertex
   $v_{u_1}$, and $\{u_2,u_3\}$ are mapped to $G'_{12}$ on $v_{u_2},v_{u_3}$. Similar as case~\ref{n-2-2},
   we can assume there are color assignments $cl_{11},cl_{12}$ to
   graphs $G'_{11},G'_{12}$ respectively. And $cl_{11}(v_1)\not=cl_{11}(v_i)$, $cl_{12}(v_i)\not=cl_{12}(v_{n_1})$,
   $cl_{11}(v_i)=cl_{12}(v_i)$, and $cl_{11}(v_1)\not=cl_{12}(v_{n_1})$. Also by defining clusters $cs_1=\{v_1,v_{u_1},v_i\}$
   and $cs_2=\{v_{u_2},v_{u_3},v_{n_1}\}$, and mapped to a simple cycle graph with $3$ vertices, then by induction,
   we can assume $cl_{11}(v_1)\not=cl_{11}(v_{u_1})\not=cl_{11}(v_i)$, and $cl_{12}(v_i)\not=cl_{12}(v_{u_2})\not=cl_{12}(v_{u_3})$.
   Next we need to show color assignments on $v_{u_1},v_{u_2},v_{u_3}$ are different from each other. Without losing generality,
   assume $cl_{11}(v_1)=2,cl_{11}(v_i)=cl_{12}(v_i)=1,cl_{12}(v_{n_1})=3$. Then by color projection, $cl_{11}(v_{u_1})$ can be
   $3$ or $4$. There are several cases here, and we discuss one here as others are much simpler:
   $cl_{11}(v_{u_1})=3$ and $cl_{12}(v_{u_2}),cl_{12}(v_{u_3})$ are $\{3,4\}$. By doing color projection with $3,4$ with $cl_{12}$,
   there is $cl_{12}(v_{n_1})=4$. Then do a projection with $2,3$, $cl_{12}(v_{u_2}),cl_{12}(v_{u_3})$ are $\{2,4\}$ which is
   different from $cl_{11}(v_{u_1})=3$. \textit{Here we can see that allowing $4$ colors on such an outerplanar graph gives extra
   flexibility, and this is the key why this Theorem exists.}
\end{enumerate}
 If $n_2>3$, then assume $u_i$ is the cut vertex of $G'_2$, and there are two graphs $G'_{21}$ and $G'_{22}$
by cutting $G'_2$ at $u_i$. We can define clusters $cs_1,cs_2$ according to the mapping and $G'_{21},G'_{22}$.
Besides that, we define a simple cycle graph $G_{u_1,u_i,u_{n_2}}$ on vertices $\{u_1,u_i,u_{n_2}\}$.
By the mapping, we can define cluster $cs_{u_1,u_i,u_{n_2}}$ in $G_1$. Easy to see we have a series of
clusters now: $cs_{u_1,u_i,u_{n_2}},cs_1,cs_2$. And by Lemma~\ref{outerplanar-complete-solution-set},
$G_1$ has a complete solution set to those clusters. By induction, in the complete solution space,
pick $ks_{u_1,u_i,u_{n_2}}$ for $cs_{u_1,u_i,u_{n_2}}$, and $ks_{u_1,u_i,u_{n_2}}\bigcap ss_{u_1,u_i,u_{n_2}}\not=\emptyset$
where $ss_{u_1,u_i,u_{n_2}}$ is the solution space for graph $G_{u_1,u_i,u_{n_2}}$. Similarly as above case,
we can get $ks_1,ks_2$ for $cs_1,cs_2$ respectively and finally get a $cl$ to $G_1$ which has overlapping
with $ss_{u_1,u_i,u_{n_2}},ss_1,ss_2$ which are solution spaces to graphs $G_{u_1,u_i,u_{n_2}},G'_{21},G'_{22}$
respectively. Hence $cl$ can be a solution to graph $G_2$. So $G_1,G_2$ has overlapping in their solution space.
\end{enumerate}
\qed
\end{proof}

\begin{corollary}\label{kernel-graph}
If a graph $G$ can be decomposed into two outerplanar graphs, then it can be colored with $\leq 4$ colors.
\end{corollary}
\begin{proof}
Straightforward from Theorem~\ref{outerplanar-graph-solution-space-overlapping}.
\qed
\end{proof}
This kind of graph is very important in planar graph, and we call such kind of graph as the \textit{kernel}
of planar graph.

\begin{corollary}\label{kenel-solution-space-overlapping}
Given two arbitrary clusters $cs_1,cs_2$ and an arbitrary mapping, two arbitrary kernel solution spaces $ks_1,ks_2$
belonging to $cs_1,cs_2$ respectively, there is $ks_1\bigcap ks_2\not=\emptyset$.
\end{corollary}
\begin{proof}
By Definition~\ref{kernel-cluster-solution-space}, there are two outerplanar graphs $G_1,G_2$ according $ks_1,ks_2$ respectively.
And $ss_1=ks_2, ss_2=ks_2$ where $ss_1,ss_2$ are solution spaces to $G_1,G_2$ respectively.
With the given mapping, by Theorem~\ref{outerplanar-graph-solution-space-overlapping}, $G_1,G_2$ have overlapping in
their solution space. Hence $ks_1\bigcap ks_2\not=\emptyset$.
\qed
\end{proof}

\section{To Prove Four Color Theorem}\label{sec-to-prove-4ct}

\begin{theorem}\label{strenghed-4ct}
In a planar graph $G(V,E)$, given a perimeter trace and an arbitrary series of clusters on it,
the clusters have a complete solution space using up to $4$ colors.
\end{theorem}
\begin{proof}
Prove this by doing induction on $|G|$.
When $|G|=1$, it is trivial. Actually, if $G$ is a simple cycle graph, then by Lemma~\ref{outerplanar-complete-solution-set},
the conclusion holds.

When $|G|>1$, we can assume the graph is not a simple cycle graph. On $G$'s an arbitrary perimeter trace
$TS=\{v_1,v_2,\ldots,v_k\}$,
and a series of cluster on $TS$, we will show that there is a complete solution space on those clusters.
As $G$ is not a simple cycle graph, we can assume there is a path $P_{v_i,v_j}$ from $v_i$ to $v_j$ where
$\{v_i,v_j\}\subset TS$ and does not cross with $TS$ on other nodes.
By splitting on this path, we can get two graphs $G_1,G_2$. So we can apply induction on those two graphs.
Name the $TS_1,TS_2$ as $TS$ on $G_1,G_2$ respectively. Define $TS_1'=TS_1\bigcup P_{v_i,v_j}$, and
$TS_2'=TS_2\bigcup P_{v_i,v_j}$. Define a series of clusters on $TS_1'$ as:
if a cluster $cs\bigcap TS_1\not=\emptyset$, then define a cluster $cs_1'=cs\bigcap TS_1$.
By Definition~\ref{series-clusters}, we will get a series of clusters on $TS_1'$. Additionally,
we define one more cluster into the series $cs=P_{v_i,v_j}$. Similarly, we can get another series of
clusters on $TS_2'$.

By doing induction on those two series of clusters, there are complete solution space
on both of them. So there is a kernel solution space on $cs=P_{v_i,v_j}$ in $TS_1'$,
and a kernel solution space on $cs=P_{v_i,v_j}$ in $TS_2'$.
By Corollary~\ref{kenel-solution-space-overlapping}, the two kernel solution spaces have non-empty
intersection so the color assignments for $G_1,G_2$ can be merged together to be a color assignment
for $G$. And by Lemma~\ref{kernel-solution-merge}, if a cluster is split into two clusters on
$TS_1',TS_2'$, their two kernel solution spaces can be merged to be a kernel solution space.
So the conclusion holds.

\qed
\end{proof}

Hence by Theorem~\ref{strenghed-4ct}, we have proved four color
theorem, and conclude as a corollary as below.
\begin{corollary}\label{4ct}
Every planar graph is $4$ colorable.
\end{corollary}

\section{Conclusion}\label{sec-conclusion}

The proof in this paper can be treated as a generalization of proof
in~\cite{WeiyaNote3}. In this paper, we have proved four color
theorem, but properties of planar graph are utilized, hence can not
be generalized to prove Hadwiger Conjecture. However
in~\cite{WeiyaNote1,WeiyaNote2,WeiyaNote3} a bunch of results can be
used to prove condition Hadwiger Conjecture when $k=5$ without using
property of planar graph. {\it Hence, we claim the ideas and
conclusions in this paper can be generalized to prove Hadwiger
Conjecture.}

\title{An Equivalent Statement of Hadwiger's Conjecture when $k=5$}

\author{Weiya Yue\inst{1}, Weiwei Cao\inst{2}}
\institute{$^1$ Google Inc, 1600 amphitheatre pkwy mountain view, CA, US 94043\\
$^2$ State Key Laboratory of Information Security,
Institute of Information Engineering, Chinese Academy of Sciences,
Beijing 100093, China\\
}
\maketitle

\begin{abstract}
Hadwiger's conjecture states that if a graph has no $K_k$ minor, then
its chromatic number is $k-1$.
In this paper, we study Hadwiger's conjecture when $k=5$ and give two
new results. First we show that a $5$-chromatic graph $G$ with no
$K_5$ minor can be reduced by minor actions to be a new $5$-chromatic graph $G'$
having minimum vertex degree $\geq 5$ and no admissive cut set;
further if $G$ can not be reduced by minor actions to be a smaller
$\geq 5$-chromatic graph and $G$ has a $5$-degree vertex,
its neighbors form a five-sided polygon. Second, we prove the
equivalence of Hadwiger's conjecture when $k=5$ with the statement that in a
$4$-chromatic graph there is a $K_4$ minor on a subset of
vertices. All conclusions we give can be generalized to arbitrary $k$ on
Hadwiger's conjecture. Those conclusions can be used to
prove existence of special structures and can greatly
simplify the proof of Hadwiger's conjecture for $k=4$. So it is promising to use such results to
prove the Hadwiger's conjecture for $k=5$, i.e. four color theorem.
\end{abstract}

\section{Introduction}

In graph theory, the Hadwiger's conjecture~\cite{HH1943} states that
if all colorings of an undirected graph $G$ need $\geq k$ colors,
then $G$ has a $K_k$ minor. When $k=5$, this conjecture is
equivalent with the four color problem which states that every planar
graph has a chromatic number
$4$~\cite{Wagner1937,Halin1964,Halin1967,Ore1967,Young1971}, i.e., a
$5$-chromatic graph has a $K_5$ minor.

The four color theorem has been proved assisted by computer for the
first time in 1976 by Kenneth Appel and Wolfgang
Haken~\cite{Appel1977-1,Appel1977-2}. Afterwards, Robertson,
Sanders, Seymour, and Thomas~\cite{Robertson1996,Robertson1997} gave a
simpler proof in 1997. Georges
Gonthier~\cite{George2005,George2008} again proved using general
purpose theorem proving software. All these good work introduced ways
to study four color theorem but used unavoidability configures of planar graph.
This property makes them difficult to be generalized to prove Hadwiger's
conjecture. It seems new idea or method is needed to continue the
research since Hadwiger's conjecture plays a very important role in graph theory.


In this paper, we study Hadwiger's conjecture for $k=5$
and give our new results. All our work try to bear the principle that they
are easy to be generalized to any $k$ of Hadwiger's
conjecture. At first we show that a $5$-chromatic graph
with no $K_5$ minor can be reduced by
applying minor actions to be a graph that has a chromatic number $\geq 5$, minimum
vertex degree $\geq 5$ and no admissive cut set.
This conclusion looks similar but in fact different from
Dirac's results in~\cite{Ore1967} which were achieved by using critical
graph. The example in Figure~\ref{fig-1}
evidently shows the reduced graph by our our result is not the one that
reduced by Dirac's result. Also, our proof avoid the use of critical graph
and knowledge of alternative path which
makes the proof simpler.


Our other contribution is that we prove that finding a $K_5$ minor in a
$5$-chromatic graph is equivalent to finding a $K_4$ minor on
a set of special vertices of its chromatic number $4$ subgraph.
Because our proof does not depend on any property of planar graphs,
this result can be used to give a restatement of Hadwiger's
conjecture. The restatement is not hard to obtain, but with a big benefit that
it can simplify the proof of Hadwiger's conjecture. In order to show this, we reprove Hadwiger's conjecture
when $k=4$ as in Theorem~\ref{Dirac-k-4-theorem}. Also, by using the
new statement, we can show existence of some simple structures like
simple cycle and forest as stated in
Theorem~\ref{related-with-pre-conclusion}. These issues are
discussed in Section~\ref{sec-conjecture}.


\section{Terminology and Definitions}\label{sec-Terminology}
In this section, necessary terminologies and definitions are
introduced.

\begin{definition}\label{color-assignment}
A color assignment to a graph $G(V,E)$ is a set of partitions of $V$,
in which each partition is an independent set and different partitions
assigned with different colors.
\end{definition}

In this paper, we use $cl$ to denote a color assignment and integers
to represent colors. Then we can say there is a $l$-color assignment $cl=\{1,2,...,l\}, |cl|=l$.
Given a graph $G=(V,E)$ and a color assignment $cl$. Function
$color_{cl}(v)$ represents vertex $v$'s color in $cl$. Without
confusion, e.x. talking only one color assignment, we sometimes skip the subscript $cl$. Given a set of
vertices $W\subset V$, $color(W)$ means all the colors used on $W$
in a color assignment. To a color assignment $cl$, define its
$frequency-vector$, abbreviated as $fv$, as
$<times_{cl}(l),times_{cl}(l-1),...,times_{cl}(1)>$, where $times$
means how many times a color is used in $cl$. Then we can compare
two color assignments by their $frequency-vector$ in lexicographical
order.

If color assignments are ordered by their $frequency-vectors$ in
lexicographical order, there exists a color assignment with the
minimum $frequency-vector$. Name this color assignment as $CL$ and
the corresponding $frequency-vector$ as $NV$. Below without explicit
explanation, we always use the color assignment whose
$frequency-vector$ is minimum.

\begin{definition}\label{kernel-vertices}
In a chromatic number $k$ graph $G(V,E)$ given a vertex set $U\subseteq V$, if in every
$k$-color assignment of $G$ $U$ are assigned with $k$
colors, then $U$ is called a set of kernel vertices of $G$.It is trivial
that $V$ is always a set of kernel vertices of $G(V,E)$.
\end{definition}

In a graph $G=(V,E)$, given a set of vertices
$S=\{s_1,s_2,...,s_x\}\subset V$ and an equivalence relation $R$ on
$S$, we use $abs_{R}(S)$ or $S_{R}$ to represent a new set of
vertices, in which every vertex is contracted from an equivalent
subset of $S$. A vertex belonging to $S_{R}$ is called super-vertex
of $R$ from $S$.
Then we can have a new graph $G_{R}=(V_{R},E_{R})$ by replacing $S$
with $abs(S)$, and $E_{R}$ is defined as below: if $s_i,s_j\in S$ are
contracted to be $s'$, then $N(s')=N(s_i)\cup N(s_j)\setminus
\{s_i,s_j\}$; edges with no endpoint in $S$ are intact.

In Figure~\ref{fig-0} there is an example to show how to do
contraction. In Figure~\ref{fig-0}.$a$, on the set of vertices
$W=\{w_1,w_2,w_3\}$ define equivalence relation
$R=\{\{w_1,w_2\},\{w_3\}\}$; in Figure~\ref{fig-0}.$b$, the graph is
$G_{R}$, in which $w'$ is contracted from $\{w_1,w_2\}$. {\it An
equivalence relation is called ``{\it admissive}'' if in the
relation, $v_1,v_2$ are equivalent implies no edge $e(v_1,v_2)$
between them}.
The relation in Figure~\ref{fig-0} is admissive.

\begin{figure}[t]
{\scriptsize
\begin{center}
\setlength{\unitlength}{0.7pt}
\begin{picture}(400,80)(-200,-35)
\put(-168,25){\oval(14,14)} \put(-168,25){\makebox(0,0){$u_1$}}
\put(-120,55){\oval(14,14)} \put(-120,55){\makebox(0,0){$w_1$}}
\put(-120,25){\oval(14,14)} \put(-120,25){\makebox(0,0){$w_2$}}
\put(-120,-5){\oval(14,14)} \put(-120,-5){\makebox(0,0){$w_3$}}
\put(-72,25){\oval(14,14)} \put(-72,25){\makebox(0,0){$u_2$}}
\put(-161,25){\line(1,1){33}} \put(-161,25){\line(1,0){33}} \put(-161,25){\line(1,-1){33}}
\put(-120,2){\line(0,1){16}}
\put(-79,25){\line(-1,1){33}}
\put(-79,25){\line(-1,0){33}}
\put(-79,25){\line(-1,-1){33}}
\put(-118,-40){\makebox(0,0){(a)}}
\put(-30,25){\makebox(0,0){$\Longrightarrow$}}
\put(15,25){\oval(14,14)} \put(15,25){\makebox(0,0){$u_1$}}
\put(63,55){\oval(18,18)} \put(63,55){\makebox(0,0){$w_{12}$}}
\put(63,-5){\oval(14,14)} \put(63,-5){\makebox(0,0){$w_3$}}
\put(111,25){\oval(14,14)} \put(111,25){\makebox(0,0){$u_2$}}
\put(22,25){\line(1,1){33}} \put(22,25){\line(1,-1){33}}
\put(63,2){\line(0,1){43}}
\put(104,25){\line(-1,1){33}}
\put(104,25){\line(-1,-1){33}}
\put(63,-40){\makebox(0,0){(b)}}
\end{picture}
\end{center}
} \caption{Examples of Reductions According to Certain Admissive
Relation} \label{fig-0}
\end{figure}

\begin{observation}\label{cover-back}
After defining equivalence relation $R$ on a set of vertices $S$, if
$R$ is an admissive relation, then every color assignment
$cl_R$ of $G_{R}$ can be extended to be a color
assignment $cl$ of $G$, where $color_{cl}(v)=color_{cl_R}(v)$ for $v\in
V\setminus S$ and for each equivalence class $S_i\subseteq S$, if $s_j\in S_i$
$color_{cl}(S_i)=color_{cl_R}(s_j)$.
\end{observation}
\begin{proof}
Since $R$ is admissive, then in $S_i$ every two vertices are not
connected by an edge. Hence we can set $color_{cl}(V\setminus
S)=color_{cl'}(V\setminus S)$, and for every vertex $w\in S_i$,
$color_{cl}(w)=color_{cl'}(s_i)$ directly.
\qed
\end{proof}

\begin{observation}\label{strengthened-graph}
In a $k$-chromatic graph $G(V,E)$ with $S\subseteq V$, if $R$ is
an admissive equivalence relation on $S$, then graph $G_R$ has a
chromatic number $\geq k$.
\end{observation}
\begin{proof}
It follows from Observation~\ref{cover-back} directly.
\qed
\end{proof}

A minor action is contracting an edge, deleting an edge, or deleting an
isolated vertex. An undirected graph $H$ is a minor of another undirected graph $G$
if a graph isomorphic to $H$ can be obtained from $G$ by
applying minor actions.

\begin{definition}\label{admissive-cut-set}
Let $G(V,E)$ be a graph with a cut set $W\subseteq V$ and admissive
equivalence relation $R$ on $W$. If graph $G'=G\setminus W$
has two subgraphs $C_1,C_2$ which are disconnected from each
other, and in graphs $C_i\cup W(i\in \{1,2\})$
a clique $abs_R(W)$ can be achieved by applying minor
actions. Then $W$ is called a admissive cut set of $G$ to $R$.
\end{definition}

Corresponding to a different minor action, we define an extension of a set of
vertices as below.
\begin{definition}\label{extending} 
Given a graph $G(V,E)$ and a set of vertices
$U\subseteq V$, if $G'(V',E')$ is reduced from $G$ by applying a
minor action, the extension of $U$, $U'\subseteq V'$ is defined
as:i) if deleting a vertex $v_1\in U$:
$U'=\{U\setminus v_1\}\cup N(v_1)$; ii) if $v_1\in U$ or $v_2\in U$
and contracting $v_1,v_2\in V$ with $e(v_1,v_2)$ to be $v'$,
$U'=\{U\setminus \{v_1,v_2\}\}\cup v'$; iii) otherwise, $U'=U$.
\end{definition}
If $H$ is obtained from $G$ by sequential minor actions, the extension of $U$
in $H$, $Ex_{H}(U)$ is defined iteratively. When it will not cause confusion, we simply
write $Ex(U)$. Trivially $Ex_{G}(U)=U$.


Conventionally, given a graph $G=(V,E)$, a vertex $v\in V$, a set of
vertices $W\subseteq V$, and a subgraph $G_s=(V_s,E_s)$ of $G$, a
new subgraph $G_s'(V_s',E_s')=G_s\cup v$ means that $V_s'=V_s\cup v$
and $E_s'=E_s\cup \{edges\ between\ v\ and\ V_s\ in\ G\}$; a new
subgraph $G_s'(V_s',E_s')=G_s\cup W$ means that $V_s'=V_s\cup W$ and
$E_s'=E_s\cup \{edges\ between\ W\ and\ W\cup V_s\ in\ G\}$. $G_W$
represents the subgraph $G\cap W$.

\begin{definition}\label{minor-on-vertices}
In a graph $G(V,E)$, given a vertex set $U\subseteq V$ and $|U|=x>0$, we say in $G$ there is
a $K_x$ minor on $U$ under the following condition holds: if a vertex $v\in K_x$ is contracted
from $S\subseteq V$, then $U\cap S\neq \emptyset$. If
$U'\subseteq V$ and $U\subseteq U'$, we also say there is a $K_x$
minor on $U'$.
\end{definition}

To prove the following Theorem~\ref{equivalence-of-4-color-theorem} in Section~\ref{},
we will show that a $K_4$ minor on $Ex(U)$ implies
a $K_5$ minor on $v\cup N(v)$ where $U=N(v)$.

Assume $c_1,c_2$ are two colors, define $f(c_1,c_2)$ as the two color exchange
function by exchanging colors $c_1$ and $c_2$. $cl_f$ is the new
color assignment by applying $f$ on the color assignment $cl$.
For convenience, sometime we use $f$ to represent one color exchange
function instead of $f(c_1,c_2)$.

\begin{observation}\label{property-of-color-exchange}
If $f$ is a color exchange function, then $f\circ f=e$; if
$f=f_1\circ f_2\circ ...f_t$ and $f'=f_t\circ...f_2\circ f_1$, then
$f\circ f'=e$. Here $e$ means the identity function.
\end{observation}
\begin{proof}
The proof is straightforward.

\qed
\end{proof}

\begin{definition}\label{minimal-5-degree-graph}
A minimal $5$-chromatic graph is a $5$-chromatic graph that can not be
reduced by minor actions to a smaller $5$-chromatic graph.
\end{definition}

\section{A New Method To Reduce Graphs}\label{sec-4-color-theorem}

In this section, we will show that if a $5$-chromatic graph
$G(V,E)$ cannot be reduced by minor actions to be a smaller graph whose
chromatic number is $\geq 5$, then it is at least $4$-connected; and it has a subgraph $G_s=(V_s,E_s)$
with vertex $v\in V_s$ such that $G'=G_s\setminus v$, which is a $3$-connected and $4$-chromatic graph and $N(v)$ is
a set of kernel vertices of $G'$. In this paper, when we
say a graph is $n$-connected, it means the graph is connected after
removing arbitrary $(n-1)$ vertices.

A $5$-chromatic graph $G=(V,E)$ has a color assignment using
only $5$ colors or $\chi(G)=5$. Name the set of vertices which are assigned color
$5$ as $V_5$.

\begin{lemma}\label{reduce-v-5}
In a $5$-chromatic graph $G=(V,E)$, given its minimum color assignment $CL$
with the $frequency-vector$ $NV$, if $v\in V_5$, then graph $G\setminus \{V_5 \setminus \{v\}\}$ has
a chromatic number $5$.
\end{lemma}
\begin{proof}
Assume the new graph has a chromatic number $\leq 4$, then there exists a
color assignment $cl$ using $\leq 4$ colors. W.L.O.G, assume colors $\{1,2,3,4\}$ are used.
Then in graph $G$, we can extend $cl$ by coloring vertices $V_5\setminus \{v\}$ with
color $5$. So in the new color assignment its $num(5)$ of the $frequency-vector$ is one less than $num(5)$ of $NV$.
This is a contradiction with that $NV$ is minimum.
\qed
\end{proof}

Given Lemma~\ref{reduce-v-5} we will henceforth assume $|V_5|=1$ in $CL$. Then if
we have a vertex $v\in V_5$, then we can assume $V_5=\{v\}$.

\begin{lemma}\label{v's-neighbors}
If $G(V,E)$ has $\chi(G)=5$ and $V_5=\{v\}$, then $N(v)$ is a set of
kernel vertices of $G'=G\setminus V_5$.
\end{lemma}
\begin{proof}
Assume there is a color assignment $cl$ on $G'$ using $4$ colors
$\{1,2,3,4\}$ and $N(v)$ are colored with $3$ colors
$\{a,b,c\}$, then we can extend $cl$ to be a $4$-color assignment for
graph $G'\cup v$ by assigning $v$ with color $\{1,2,3,4\}\setminus
\{a,b,c\}$. This contradicts Lemma~\ref{reduce-v-5}.
\qed
\end{proof}

\begin{lemma}\label{graph-with-cut-set-clique}
Suppose $W$ is a cut set of a graph $G(V,E)$, the subgraph
$G_W=G\cap W$ is a clique, and $C_l,C_r$ are subgraphs of
graph $G'=G\setminus W$ and they are disconnected from each other. If
$G_l=C_l\cup W, G_r=C_r\cup W$ have color assignment $cl_l,cl_r$
respectively, then $G$ has a color assignment $cl$ satisfying
$color_{cl}(G)\leq max(color_{cl_l}(G_l), color_{cl_r}(G_r))$.
\end{lemma}
\begin{proof}
Because the subgraph $G_W$ is a clique, all vertices $\in W$ have different colors from each other. By color
exchanging we can assume $color_{cl_l}(W)=color_{cl_r}(W)$. Hence
$cl_l,cl_r$ can be combined to be one color assignment for $G$ without introducing more colors.
\qed
\end{proof}

\begin{lemma}\label{graph-with-cut-set-which-can-be-clique}
In graph $G(V,E)$, suppose $W$ is a admissive cut set of $G$, the
corresponding admissive equivalence relation on $W$ is $R$, and the
two subgraphs in $G'=G\setminus W$ are $C_l,C_r$. If $G_l=C_l\cup
abs_R(W), G_r=C_r\cup abs_R(W)$ both have chromatic numbers $\leq k$,
then $G$ has a chromatic number $\leq k$.
\end{lemma}
\begin{proof}
Because $G_l, G_r$ both have chromatic numbers $\leq k$, there are
color assignments $cl_l,cl_r$ using $\leq k$ colors for $G_l, G_r$
respectively. By Definition~\ref{admissive-cut-set}, $abs_R(W)$ is
a clique, hence we can set graph $G_R=G_l\cup G_r$, and by
Lemma~\ref{graph-with-cut-set-clique}, $cl_l,cl_r$ can be combined
to be one color assignment $cl$ for $G_R$ using $\leq k$ colors. By
Observation~\ref{cover-back}, $cl$ can be extended to be a color
assignment $cl'$ for graph $G$, and $cl'$ use $\leq k$ colors. So
$G$ has a chromatic number $\leq k$.
\qed
\end{proof}

\begin{theorem}\label{no-independent-set-cut-vertices}
Given a $5$-chromatic graph $G(V,E)$, suppose $W$ is a minimal cut
set of $G$; setting $G'=G\setminus W$, assume $C_l,C_r$ are two
subgraphs of $G'$ which are disconnected from each other. Assume $W$ is
a admissive cut set of $G$. If there is no $K_5$ minor in $G$,
then at least one of $G'_{l}=abs(W)\cup C_l, G'_{r}=abs(W)\cup
C_r$ has its chromatic number $\geq 5$.
\end{theorem}
\begin{proof}
By Definition~\ref{admissive-cut-set}, $abs_R(W)$ is a clique. If
$abs_{R}(W)$ is a $\geq 4$ clique, a $K_5$ minor can be constructed
immediately. If not, $abs_R(W)$ is a $i\leq 3$ clique. Suppose
$G'_l,G'_r$ both have chromatic numbers $\leq 4$, by
Lemma~\ref{graph-with-cut-set-which-can-be-clique}, $G$ has its
chromatic number $\leq 4$ which is a contradiction with assumption.
Hence, at least one of graphs $G'_{l}, G'_{r}$ has its chromatic
number $\geq 5$.
\qed
\end{proof}

Theorem~\ref{no-independent-set-cut-vertices} can be generalized to
different $k$ of Hadwiger's conjecture.

\begin{corollary}\label{removing-one-vertex-to-be-3-connected}
Assume $G=(V,E)$ with $\chi{X}(G)=5$. If $G$ is not at
least $4$-connected, then $G$ can be reduced by minor actions to be a
smaller graph with chromatic number $\geq 5$.
\end{corollary}
\begin{proof}
By the proof of Theorem~\ref{no-independent-set-cut-vertices}, we only need to
show a cut set in $G$ whose cardinality $\leq 3$ is admissive.
Assume in $G$ there is a minimum cut set $W$ with $|W|\leq 3$. It is
easy to see that there always exists an admissive equivalence
relation$R$ on $W$ with $abs_{R}(W)$ as a clique. Assume $C_l,C_r$
are two subgraphs of $G'=G\setminus W$ which are disconnected from
each other. From simple analysis of cases depending on $G\cap W$, by
doing minor actions on $C_l$, $abs_{R}(W)$ can be achieved; similarly for
$C_r$. One case of $|W|=3$ is displayed in Figure~\ref{fig-1}.$a$,
in which $W=\{w_1,w_2,w_3\}$, and say $u\in C_l$. Because $W$ is a
minimum cut set in $G$, $u$ can connect with $W$ along independent three paths
shown as dash-lines in the figure. There are two ways to define
admissive relations on $W$ and either one is sufficient to prove our
conclusion:
\begin{enumerate}
\item $R=\{\{w_1,w_3\},\{w_2\}\}$ 
By applying minor actions, $w_1,w_3$ can be contracted along pathes
$P_{w_1,u}$ and $P_{u,w_3}$ to be vertex $w'$, then we get the
$abs_{R}(W)$ as a $2$-clique.
\item $R=\{\{w_1\},\{w_2\},\{w_3\}\}$.
Simply by contracting $\{w1,u\}$, we can get the $abs_{R}(W)$ as a $3$-clique.
\end{enumerate}
All other cases can be analyzed similarly. Hence $W$ is a admissive
cut set of $W$.
\qed
\end{proof}

\begin{corollary}\label{cor-minimum-degree-5}
In a $5$-chromatic graph $G=(V,E)$, if $G$ has
minimum vertex degree $< 5$ then either $G$ has a $K_5$ minor or
$G$ can be reduced by minor actions to be a smaller graph with a chromatic number $\geq 5$.
\end{corollary}
\begin{proof} 
By Corollary~\ref{removing-one-vertex-to-be-3-connected}, $G$ is at
least $4$-connected. So we only need to discuss the case when there
is a vertex $v\in V$ with degree $4$, then in this case
$W=N(v)=(v_1,v_2,v_3,v_4)$ is a minimum cut set.

If $G\cap N(v)$ is a $4$-clique, then $v\cup N(v)$ is a $5$-clique
and hence a $K_5$ minor. Otherwise, if there is no $K_5$ minor in $G$,
then W.L.O.G. we can assume there is no edge $e(v_1,v_2)$ where
$v_1,v_2\in N(v)$. Define the equivalence relation $R$ on $N(v)$ as:
$\{\{v_1,v_2\},{v_3},{v_4}\}$. 

Then we have $|abs_{R}(W)|<4$ and by applying minor contractions
only on $v\cup N(v)$, $abs_{R}(W)$ can be achieved. Suppose
$G'=\{G\setminus \{v\cup N(v)\}\}\cup abs_{R}(W)$ has a chromatic
number $<5$, then there is a color assignment $cl'$ of $G'$ using
$<5$ colors. Since $abs_{R}(W)$ can be colored with $<4$ colors,
hence $cl'$ can be extended to be a color assignment of graph
$G_{R}$ using $<5$ colors. By Observation~\ref{strengthened-graph},
we can have a color assignment for $G$ using $<5$ colors which is a
contradiction. Hence the chromatic number of $G'$ has to be $\geq
5$.
\qed
\end{proof}

\begin{figure}[t]
{\scriptsize
\begin{center}
\setlength{\unitlength}{0.7pt}
\begin{picture}(400,80)(-200,-35)
\put(-170,55){\oval(14,14)} \put(-170,55){\makebox(0,0){$w_1$}}
\put(-120,55){\oval(14,14)} \put(-120,55){\makebox(0,0){$w_2$}}
\put(-70,55){\oval(14,14)} \put(-70,55){\makebox(0,0){$w_3$}}
\put(-120,-5){\oval(14,14)} \put(-120,-5){\makebox(0,0){$u$}}
\put(-163,55){\line(1,0){36}} \put(-113,55){\line(1,0){36}}
\put(-120,-5){\dashline{3}(-5,5)(-44,56)}
\put(-120,-5){\dashline{3}(0,7)(0,53)}
\put(-120,-5){\dashline{3}(5,5)(44,56)}
\put(-120,-40){\makebox(0,0){(a)}}
\put(-20,55){\oval(14,14)} \put(-20,55){\makebox(0,0){ $w_1$ }}
\put(10,55){\oval(14,14)} \put(10,55){\makebox(0,0){ $w_2$ }}
\put(40,55){\oval(14,14)} \put(40,55){\makebox(0,0){ $w_3$ }}
\put(70,55){\oval(14,14)} \put(70,55){\makebox(0,0){ $w_4$ }}
\put(100,55){\oval(14,14)} \put(100,55){\makebox(0,0){ $w_5$ }}
\put(130,55){\oval(14,14)} \put(132,55){\makebox(0,0){$w_6$ }}
\put(55,20){\oval(14,14)} \put(55,20){\makebox(0,0){$u_1$}}
\put(55,-15){\oval(14,14)} \put(55,-15){\makebox(0,0){$u_2$}}
\put(-13,55){\line(1,0){17}} \put(107,55){\line(1,0){16}}
\put(55,15){\dashline{2}(-7,7)(-40,33)}
\put(55,15){\dashline{2}(-7,7)(-16,33)}
\put(55,15){\dashline{2}(7,7)(16,33)}
\put(55,15){\dashline{2}(7,7)(40,33)}
\put(55,-15){\dashline{2}(-7,7)(-75,62)}
\put(55,-15){\dashline{2}(7,7)(76,62)}
\put(56,-40){\makebox(0,0){(b)}}
\end{picture}
\end{center}
} \vspace*{1mm} \caption{Examples of Reductions According to Certain
Admissive Relation} \label{fig-1}
\end{figure}

Results from Dirac similar to
Corollary~\ref{removing-one-vertex-to-be-3-connected} and
Corollary~\ref{cor-minimum-degree-5} are collected in~\cite{Ore1967}
by working on critical graph.
In fact our conclusions are different.
For example, in Figure~\ref{fig-1}.$b$, assume
$W=\{w_1,w_2,w_3,w_4,w_5,w_6\}$ is a cut set and $u_1,u_2$ are in
one component $C_l$ in $G'=G\setminus W$. With results
of~\cite{Ore1967}, this graph can not be contracted to be a smaller
graph with a chromatic number $\geq 5$. But we can define an admissive
relation $R$ on $W$ by putting
$R=\{\{w_1,w_6\},\{w_2,w_3,w_4,w_5\}\}$, then by contracting paths
passing through $u_1$ and $u_2$ displayed in dash-line in
Figure~\ref{fig-1}.$b$, $abs_{R}(W)$ can be achieved as a
$2$-clique. If we can do the same thing in $G'\setminus C_l$, then
by Theorem~\ref{no-independent-set-cut-vertices} $G$ can be
contracted to be a smaller graph $G_R$ with chromatic number $\geq
5$. Further, we have that:

\begin{proposition}\label{minimum-degree-5-condition}
In a $5$-chromatic graph $G=(V,E)$, assume $v\in V$ has degree
$5$. If $N(v)$ is not a five-sided polygon, then $G$ has a $K_5$
minor or $G$ can be reduced by applying minor actions to be a
smaller graph with chromatic number $\geq 5$.
\end{proposition}
\begin{proof}
Assume $N(v)=\{v_1,v_2,v_3,v_4,v_5\}$ and $G$ can not be reduced. By
Corollary~\ref{removing-one-vertex-to-be-3-connected}, $G'=G\setminus
v$ is at least $3$ connected. First assume there is a triangle on $N(v)$,
say on $\{v_1,v_2,v_3\}$, then from $v_4$(or $v_5$) to the
triangle there are at least $3$ disjoint paths. Hence on
$\{v_1,v_2,v_3,v_4\}$ there is a $K_4$ minor without using $v$, then
combined with $v$, there is a $K_5$ minor. Now assume that there is no triangle
on $N(v)$. If there are three independent vertices in $N(v)$, say $\{v_1,v_2,v_3\}$, we
may define the equivalence relation $R=\{\{\gamma_1,\gamma_2,\gamma_3\},\gamma_4,\gamma_5\}$.
Following the proof of Corollary~\ref{cor-minimum-degree-5}, we can see that
$G$ can be reduced, a contradiction. Therefore the subgraph on $N(v)$ can only be a five-sided polygon.

\qed
\end{proof}

From the reduction of Proposition~\ref{minimum-degree-5-condition}, it does not change
the planarity of a graph. Hence Proposition~\ref{minimum-degree-5-condition} can be
used to prove the Five Color Theorem in
a different way by using the fact that a planar graph has its minimum degree $\leq 5$ and has no
$K_5$ minor.


\begin{theorem}\label{G-prime-is-3-connected}
Given a $5$-chromatic graph $G=(V,E)$, if $G$ can not be reduced by applying minor
actions to be a smaller graph with a chromatic number $\geq 5$, then
there is a vertex $v\in V$, such that $G'=G\setminus v$ is 3-connected and $N(v)$ is a set of kernel vertices of $G'$.
\end{theorem}
\begin{proof}
If $G$ can not be reduced by applying minor actions to be a smaller
graph with a chromatic number $\geq 5$, then $|V_5|=1$ (otherwise remove vertices from $V_5$)
and we can choose vertex
$v\in V_5$. By Lemma~\ref{v's-neighbors} $N(v)$ is a set of kernel
vertices of $G'$. By
Corollary~\ref{removing-one-vertex-to-be-3-connected}, $G'$ must be
at least $3$-connected.
\qed
\end{proof}


\section{A Restatement of Hadwiger's Conjecture}\label{sec-conjecture}

In this section we conceive a new conjecture and we prove it is
equivalent with the case $k=5$ of Hadwiger's conjecture.

\begin{conjecture}\label{conjecture}
In a $3$-connected chromatic number $4$ and $K_5$ minor free graph
$G(V,E)$, if $U$ is a set of kernel vertices of $G$, then there is a
$K_4$ minor on some extension $Ex(U)$.
\end{conjecture}

Hadwiger's conjecture when $k=5$ claims the existence of $K_5$ minor
in a $5$-chromatic graph. If a graph $G$ has a subgraph
containing a $K_5$ minor, then trivially there is a $K_5$ minor in
original graph $G$. Below without explicit explanation it is assumed
that the discussed graph $G(V,E)$ has all its subgraphs are $K_5$
minor free to eliminate this trivial case.

For convenience in rest of this paper, without explicit explanation,
we assume a given chromatic number $k$ graph can not be reduced to
be a chromatic number $\geq k$ graph by applying minor actions. The
assumption is reasonable, because if not we can work on the smaller
graph.
\begin{theorem}\label{equivalence-of-4-color-theorem}
In an irreducible $5$-chromatic graph $G$, there exists a $K_5$ minor in $G$
if and only if Conjecture~\ref{conjecture} is true for $G$.
\end{theorem}
\begin{proof}
We need to prove that if Conjecture~\ref{conjecture} is true for $G$,
then $G$ has a $K_5$ minor. The reverse can
be proved similarly. Assume Conjecture~\ref{conjecture} is true for $G$.
By Theorem~\ref{G-prime-is-3-connected},
there is a vertex $v\in V$ satisfying $G'=G\setminus v$ is
3-connected and $N(v)$ is a set of kernel vertices in $G'$. By
Conjecture~\ref{conjecture}, there is a $K_4$ minor on some $Ex(N(v))=N(v)$ in
$G'$.

Next we show that the $K_4$ minor on $Ex(N(v))$ can be
combined with $v$ to be a $K_5$ minor in $G$. This can be done by
applying minor actions in graph $G$ so that $Ex(N(v))$ can be
achieved as neighbors of $v$.

From Definition~\ref{extending}, the only non-trivial case comes
from applying the action of deleting a vertex $v_1\in Ex(N(v))$.
Since at beginning it is initialized that $Ex(N(v))=N(v)$, i.e.
trivially $Ex(N(v))$ are neighbors of $v$ in $G$. When we delete a
vertex $v_1\in Ex(N(v))$ in $G'$, correspondingly in graph $G$ we
contract this vertex $v_1$ with $v$, hence $N(v_1)\setminus v$
become neighbors of $v$. By Definition~\ref{extending}, after
deleting $v_1$ in $G'$ the new $Ex(N(v))$ is updated to be
$\{Ex(N(v))\setminus v_1\} \cup N(v_1)$ which
are $v$'s new neighbors in graph $G$. Hence, the $K_4$ minor on
$Ex(N(v))$ can be combined with $v$ to be a $K_5$ minor.
\qed
\end{proof}

It is easy to see that, Theorem~\ref{equivalence-of-4-color-theorem} can
be generalized to give a new statement of Hadwiger's conjecture as
below:

\begin{corollary}\label{cor-new-statement-of-hadwiger-conjecture}
Hadwiger's conjecture when $k=x$ is correct if and only if in a
chromatic number $(x-1)$ graph, there is a $K_{x-1}$ minor on an extension of
its kernel vertices.
\end{corollary}

"the new theorem"

\begin{theorem}
For a chromatic number $k=x$ graph $G$, if it has a $k_x$ minor, then it is $x$-connected, otherwise it can be reduced by minor
action to be a smaller graph $G'$, and $G'$ has a $k_x$ minor if and only if $G$ has a $k_x$ minor.
\end{theorem}
\begin{proof}

\end{proof}

In order to show the new statement of Hadwiger's conjecture is useful,
at first we use it to prove some new properties, and then we give a
quite simple proof to Hadwiger's conjecture when $k=4$.

\begin{lemma}\label{k-3-minor-on-3-vertices}
If a graph $G(V,E)$ is 2-connected, then there is a $K_3$ minor on any
arbitrary three vertices.
\end{lemma}
\begin{proof}
Assume there are three vertices $\{v_1,v_2,v_3\}\subseteq V$,
because $G$ is 2-connected, then from $v_1$ to $v_2$ and from $v_1$
to $v_3$ there are two vertex disjoint paths $P_{1,2}, P_{1,3}$
respectively.

Similarly, because $G$ is 2-connected, from $v_2$ to $v_3$ there is a path $P_{2,3}$
which does not pass through $v_1$. It is easy to see no matter how
$P_{2,3}$ crosses with $P_{1,2},P_{1,3}$, there is a $K_3$ minor
on $v_1,v_2,v_3$.

\qed
\end{proof}

\begin{corollary}\label{cor-k-3-minor-on-3-vertices}
If graph $G(V,E)$ is 3-connected, then for any vertex $v\in V$
there is a $K_4$ minor in $v\cup N(v)$.
\end{corollary}
\begin{proof}
Because $G$ is 3-connected, $N(v)\geq 3$ and $G'=G\setminus v$ is 2-connected. Then
by Lemma~\ref{k-3-minor-on-3-vertices}, there is a $K_3$ minor on
$N(v)$. Hence, there is a $K_4$ minor in $v\cup N(v)$.
\qed
\end{proof}

For convenience below we use $@$ to represent a simple cycle.
Without confusion $@$ also represents the vertices on the cycle.
\begin{lemma}\label{k-3-minor-on-more-than-one-cycle}
If a 3-connected graph $G(V,E)$ has a vertex set $U\subseteq V$ and the sub-graph
$U\cap G$ includes a simple cycle $@$ such that $U\setminus @\not
=\emptyset$, then there is a $K_4$ on $U$.
\end{lemma}
\begin{proof}
Because $U\setminus @\not= \emptyset$, assume $v\in U\setminus @$.
Since $G$ is 3-connected, $v$ can connect with the simple cycle $@$
via three disjoint path. So $v\cup @$ can be reduced to be a $K_4$
minor.
\qed
\end{proof}

\begin{lemma}\label{a-small-conclusion}
If a graph $G(V,E)$ is 3-connected has a vertex set $U\subseteq V$, then either
\begin{enumerate}
\item there is a $K_4$ minor on $Ex(U)$; or
\item the subgraph $U\cap G$ is a simple cycle or a reduced forest, in which every tree is a path graph.
\end{enumerate}
\end{lemma}
\begin{proof}
It follows from Corollary~\ref{cor-k-3-minor-on-3-vertices} and
Lemma~\ref{k-3-minor-on-more-than-one-cycle}. Only need to notice
that if a tree is not a path graph, then there is a vertex $v$ with
$|N(v)|\geq 3$.
\qed
\end{proof}


\begin{theorem}\label{related-with-pre-conclusion}
In a $5$-chromatic graph $G(V,E)$ which can not be reduced to
be a smaller graph by minor actions, choosing vertex $v\in V_5$ and
setting $G'=G\setminus v$, in $G$ if a $K_5$ minor can be
constructed,
then 
$N(v)\cap G'$ is a simple cycle or a reduced forest, in which every
tree is a path graph.
\end{theorem}
\begin{proof}
By Theorem~\ref{G-prime-is-3-connected}, we can assume $G'$ is
3-connected. Similar as proof of Theorem~\ref{equivalence-of-4-color-theorem}, if
there is a $K_4$ minor on $Ex(N(v))$, there is a $K_5$ minor on $v\cup
N(v)$. Hence by Lemma~\ref{a-small-conclusion} either a $K_5$ minor
can be constructed, otherwise the subgraph $N(v)\cap G'$ is a simple
cycle or a reduced forest, in which every tree is a path graph.
\qed
\end{proof}


Theorem\label{Dirac-k-4-theorem} below was proved by Dirac in~\cite{G1952}, but we
give a new simpler proof to show why the new statement
of Hadwiger's conjecture in
Corollary~\ref{cor-new-statement-of-hadwiger-conjecture} can be used
to simply prove the Hadwiger's conjecture.
\begin{theorem}\label{Dirac-k-4-theorem}
If a graph $G(V,E)$ has a chromatic number $4$, then there is a $K_4$ minor.
\end{theorem}
\begin{proof}
Similar to Corollary~\ref{removing-one-vertex-to-be-3-connected},
$G$ can be assumed to be irreducible and therefore $3$-connected. So for
any $v\in V$, we have $|N(v)|\geq 3$. By
Corollary~\ref{cor-k-3-minor-on-3-vertices}, there is a $K_4$ minor
in $v\cup N(v)$.
\qed
\end{proof}

Our proof of Theorem~\ref{Dirac-k-4-theorem} is simpler and it gives
a way to find such a $K_4$ minor which is not included in Dirac's work in~\cite{G1952}.

\section{Conclusion and Next Step of Work}\label{sec-conclusion}

In this paper, we show that if a $5$-chromatic graph
can not be reduced to be a smaller chromatic $\geq 5$ graph,
and it has no $K_5$ minor, then it has no admissive cut
set (Theorem~\ref{no-independent-set-cut-vertices}) and
its minimum vertex degree is $\geq 5$ (Corollary~\ref{cor-minimum-degree-5});
further if there is a degree $5$ vertex such as in a planar graph, its neighbors are
on a five-sided polygon (Porposition~\ref{minimum-degree-5-condition}). Moreover, we give a new statement of
Hadwiger's conjecture (Theorem~\ref{equivalence-of-4-color-theorem}). By working on the new statement, to find a
$K_5$ minor in a $5$-chromatic graph is equivalent to finding
a $K_4$ minor on a chromatic number $4$ subgraph's kernel vertices,
and such kernel vertices are proved to have some special
structures (Theorem~\ref{related-with-pre-conclusion}). As an exercise we find by the new statement Hadwiger's conjecture when
$k=4$ can be proved easily and a $K_4$ minor can be constructed (Theorem~\ref{Dirac-k-4-theorem}).

In the next paper~\cite{WeiyaNote2}, we will strengthen the conclusion of
Theorem~\ref{related-with-pre-conclusion} to show that the set of
kernel vertices $N(v)$ locates on a simple cycle no matter $N(v)\cap
G'$ is a simple cycle or a forest. Also we will show that by using
the new statement of Hadwiger's conjecture, we can give a simple proof
of Wagner's equivalence theorem. Our proof is different from existing proofs
~\cite{Wagner1937,Halin1964,Halin1967,Ore1967,Young1971} since it does not depend on Kuratowski's Theorem.

\section*{Acknowledgements} I would like to thank my advisor Prof. John Franco, Prof. Dieter Schmidt and Prof. Kenneth Berman for their support and discussion. Special thanks to Prof. Gregory White for his endless discussion
and helpful comments.

\title{A New Proof of Wagner's Equivalence Theorem}

\author{Weiya Yue\inst{1}, Weiwei Cao\inst{2}}
\institute{$^1$ Computer Science Department, University of Cincinnati, Ohio, US 45220\\
$^2$ State Key Laboratory of Information Security, Graduate School
of Chinese Academy of Sciences, Beijing China 100049\\
\email{weiyayue@hotmail.com}}

\maketitle

\vspace{-5mm}
\section{Abstract}
In this paper, we prove that in a $3$-connected chromatic number $4$
graph $G(V,E)$ given a vertex set $U\subseteq V$, if there is no
$K_4$ minor on $U$, then $U$ is included by a simple cycle of $G$.
Moreover, when $G$ is a planar graph, the simple cycle is boundary
of an external face of $G$. Based on above results we give a new
proof of Wagner's Equivalence Theorem without using Kuratowski's
Theorem which is different from existing proofs.


\section{Introduction}

Hadwiger Conjecture~\cite{HH1943} states that, if an undirected
graph $G$ has chromatic number $k$, then it has a $K_k$ minor. When
$k=5$, this conjecture is a generalization of four-color problem.
The four color theorem has been proved assisted by
computer~\cite{Appel1977-1,Appel1977-2,Robertson1996,Robertson1997,George2005,George2008}.
All such proofs have difficulties in readability and checkability,
and can not be generalized to prove arbitrary $k$ of
Hadwiger Conjecture, hence it is still important to do research on
the connection between Hadwiger Conjecture when $k=5$ and four color problem in order to
find the hidden properties which may lead to a short and
generalizable proof to Hadwiger Conjecture.

In~\cite{WeiyaNote1} it concludes that to prove Hadwiger Conjecture
when $k=5$ is equivalent to prove that a $3$-connected chromatic
number $4$ graph $G(V,E)$ has a $K_4$ minor on its kernel vertices.
In this paper, we prove that in a $3$-connected chromatic
number $4$ graph $G(V,E)$ given $U\subseteq V$, either $K_4$ minor on
$U$ can be found or a subgraph of $U$ has certain and elegant structures.

By above results, we find a way to apply induction method on
graphs to prove properties of $U$. It aids us to give a new
proof of Wagner's Equivalence Theorem without using Kuratowski's
Theorem which is different from existing
proofs~\cite{Wagner1937,Halin1964,Halin1967,Ore1967,Young1971}.

In section~\ref{sec-Terminology}, terminologies and some
preliminary results are introduced. In section~\ref{simple-cycle},
it is shown that in a $3$-connected chromatic number $4$ graph
$G(V,E)$, on a vertex set $U\subseteq V$ either a $K_4$ minor can be found,
otherwise $U$ is included by a simple cycle. In section~\ref{wagner}, we give a new simple proof of Wagner's
Equivalence Theorem.

\section{Terminology Definition and Preliminary Results}\label{sec-Terminology}

In this paper conventional graph theory terminology is applied and
some definitions are quoted from~\cite{WeiyaNote1}.

\begin{definition}\label{color-assignment}
A color assignment to a graph $G(V,E)$ is a set of partitions of $V$,
in which each partition is an independent set and different partition is
assigned with a different color.
\end{definition}
We use $cl$ to denote a color assignment and integers to represent
colors. Then we can say there is a $l$-color assignment
$cl=\{1,2,...,l\}, |cl|=l$.

\begin{definition}\label{kernel-vertices}
In a chromatic number $k$ graph $G(V,E)$ given a vertex set
$U\subseteq V$, if in every $G$'s $k$-color assignment $U$ are
assigned with $k$ colors, then $U$ is called a set of kernel
vertices of $G$.
\end{definition}

Corresponding to a different minor action, we define an extension of a set of
vertices as below.
\begin{definition}\label{extending} 
Given a graph $G(V,E)$ and a set of vertices
$U\subseteq V$, if $G'(V',E')$ is reduced from $G$ by applying a
minor action, the extension of $U$, $U'\subseteq V'$ is defined
as:i) if deleting a vertex $v_1\in U$:
$U'=\{U\setminus v_1\}\cup N(v_1)$; ii) if $v_1\in U$ or $v_2\in U$
and contracting $v_1,v_2\in V$ with $e(v_1,v_2)$ to be $v'$,
$U'=\{U\setminus \{v_1,v_2\}\}\cup v'$; iii) otherwise, $U'=U$.
\end{definition}
When minor actions are applied sequentially, extension can be
defined iteratively and an iterated extension is denoted as
$Ex(U)$. $U$ is trivially an extension of itself when no action applied.

In a graph $G(V,E)$, given a vertex set $U\subseteq V$ and $|U|=x>0$, we say in $G$ there is
a $K_x$ minor on $U$ under the following condition holds: if a vertex $v\in K_x$ is contracted
from $S\subseteq V$, then $U\cap S\neq \emptyset$. If
$U'\subseteq V$ and $U\subseteq U'$, we also say there is a $K_x$
minor on $U'$.

In a simple cycle $cy$, after choosing arbitrary vertex $u\in cy$,
starting with $u$, by tracing along the cycle $cy$, a $series$ $s$
of vertices is generated. After deleting vertices from $s$, the left
series $s'$ is still called a $series$ of $cy$.

Two series $s_1,s_2$ are isomorphism if $i)$ vertex $v\in s_1$ if
and only if $v\in s_2$; and $ii)$ by rotating or reversing the
series, $s_1$ and $s_2$ can be transformed to be each other. Easy to
see, {\it any two series corresponding to one cycle are isomorphic
if and only if they contain the same set of vertices.} When we are
talking series, if every series of a cycle has a special property,
without confusion, we say the cycle has such a property.

A roundly continuous part of a series $s$ is called a $cluster$ of
$s$. An empty set of vertices can be a cluster of any series. If we
decompose $s$ into a set of clusters $cs=\{cs_1,cs_2,...,cs_x\}$ by
order where $\bigcup_{i=1}^{x} cs_i = s$. $cs$ is called clusters of
series $s$ if the end of $cs_i$ may only overlap on $\leq 1$ vertex
with the beginning of $cs_{i+1\ mod\ x}$.

For example, in Figure~\ref{fig-0}.$a$, we have cycle
$cy=\{u,c_1,c_2,c_3,c_4,c_5,u\}$, then $s_1=``c_2,c_3,c_4,c_5,u"$
and $s_2=``u,c_2,c_3,c_4,c_5"$ are $cy$'s two isomorphic series. In
series $s_1$, because {\it one cluster can be chosen roundly,
$``u,c_2,c_3"$ is one cluster.} $\{``c_3,c_4",
``c_5,u",``u,c_2,c_3"\}$ and $\{``c_3,c_4,c_5", ``u",``u,c_2,c_3"\}$
are two sets of clusters of $s_1$.

For convenience, to series $s$ of cycle $cy$ with clusters
$cs=\{cs_1,cs_2,...,cs_x\}$, we use $cy_{cs_i}$ to represent the arc
of $cy$ containing $cs_i$ and not containing the other clusters of
$cs$. With two vertices $\{c_1,c_2\}\subset cy$, we use
$cy_{c_1,c_2}, cy_{c_2,c_1}$ to represent the two split arcs of
$cy$. Also if vertex $u\not \in cy$, when we say $u$ can connect
with $cy$ on a set of vertices $CR\subseteq cy$, that means $u$ can
connect with $cy$ on $CR$ without passing through the other vertices
of $cy$. In Figure~\ref{fig-0}.$a$, set
$cy=\{c_1,c_2,c_3,c_4,c_5,c_1\}$ and $u\not \in cy$, then we say $u$
can connect with $cy$ on vertices $\{c_1,c_2,c_5\}$.
In a simple cycle $cy$, when we say along the cycle
$U=\{u_1,u_2,...,u_x\}\subseteq cy$, that means vertices in $U$ are
in the order as they appear on the cycle.

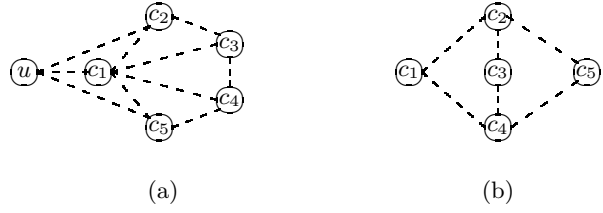
\begin{figure}[t]
\begin{center}
\setlength\unitlength{0.7pt}
\begin{picture}(400,80)(-200,-35)
\put(-173,25){\oval(14,14)} \put(-173,25){\makebox(0,0){$u$}}
\put(-133,25){\oval(14,14)} \put(-133,25){\makebox(0,0){$c_1$}}
\put(-100,55){\oval(14,14)} \put(-100,55){\makebox(0,0){$c_2$}}
\put(-62,40){\oval(14,14)} \put(-62,40){\makebox(0,0){$c_3$}}
\put(-100,-5){\oval(14,14)} \put(-100,-5){\makebox(0,0){$c_5$}}
\put(-62,10){\oval(14,14)} \put(-62,10){\makebox(0,0){$c_4$}}
\dashline{4.0}(-166,25)(-138,25) \dashline{4.0}(-166,25)(-105,50)
\dashline{4.0}(-166,25)(-105,0)
\dashline{4.0}(-126,25)(-105,50) \dashline{4.0}(-126,25)(-69,40)
\dashline{4.0}(-126,25)(-105,0) \dashline{4.0}(-126,25)(-69,10)
\dashline{4.0}(-93,55)(-67,45) \dashline{4.0}(-62,33)(-62,17)
\dashline{4.0}(-67,5)(-93,-5)
\put(-98,-40){\makebox(0,0){(a)}}
\put(35,25){\oval(14,14)} \put(35,25){\makebox(0,0){$c_1$}}
\put(83,55){\oval(14,14)} \put(83,55){\makebox(0,0){$c_2$}}
\put(83,25){\oval(14,14)} \put(83,25){\makebox(0,0){$c_3$}}
\put(83,-5){\oval(14,14)} \put(83,-5){\makebox(0,0){$c_4$}}
\put(131,25){\oval(14,14)} \put(131,25){\makebox(0,0){$c_5$}}
\dashline{4.0}(42,25)(76,55)\dashline{4.0}(42,25)(76,-5)
\dashline{4.0}(83,48)(83,32) \dashline{4.0}(83,18)(83,2)
\dashline{4.0}(90,-5)(126,20) \dashline{4.0}(126,30)(88,55)
\put(83,-40){\makebox(0,0){(b)}}
\end{picture}
\end{center}
\caption{Examples of cycle and twin-cycle} \label{fig-0}
\end{figure}

\begin{lemma}\label{one-vertex-plus-one-cycle}
In a graph $G(V,E)$, given a vertex $u\in U\subseteq V$ and a cycle $cy$ satisfying that
$u\not \in cy$ and $|\{U\}\cap cy|\geq 3$ and $u$ can connect $cy$
on a set of vertices $CR$ with $|CR|\geq 3$, if there is no $K_4$
minor on $U$, then $CR$ and $\{U\cap cy\}$ are two clusters
in any series including $CR\cup U$ on $cy$.
\end{lemma}
\begin{proof}
The proof is straightforward.

\qed
\end{proof}

\begin{definition}\label{twin-cycle}
In a graph $G(V,E)$, given a vertex st $U\subseteq V$, we say a $G$'s sub-graph
$G_s(V_s,E_s)$ is a $twin-cycle$ on $U$, if $U\subseteq V_s$ and:
\begin{enumerate}
\item there are two cycles $cy_1,cy_2$ with $cy_1\cup cy_2 = G_s$;
\item $P=cy_1\cap cy_2$ is one path with $v_1,v_2$ as the two
endpoints;
\item $U\cap \{cy_1\setminus P\}\not =\emptyset, U\cap \{cy_2\setminus P\}\not
=\emptyset$ and $U\cap \{P\setminus \{v_1,v_2\} \} \not =\emptyset$.
\end{enumerate}
We say the two vertices $v_1,v_2$ are crossing vertices of the
$twin-cycle$, and $\{cy_1\setminus P\}\cup \{v_1,v_2\}, \{cy_2\setminus
P\}\cup \{v_1,v_2\}, P$ are its three
half-cycles.
\end{definition}

{\it Here we emphasize when there is one sub-graph
$G_s$ on $U$, without explicit explanation, it is assumed
$U\subseteq cy$.} In Figure~\ref{fig-0}.$b$, there is a twin-cycle
on $\{c_1,c_3,c_5\}$ in which $P=``c_2,c_3,c_4"$ and $\{c_2,c_4\}$
are its two crossing vertices, and $``c_2,c_1,c_4",``c_2,c_5,c_4",
P$ are the three half-cycles.

We will use a simple operation named ``reforming" on simple cycle
and $twin-cycle$. Now we explain it with an example on
Figure~\ref{fig-0}.$a$.
In Figure~\ref{fig-0}.$a$, using the simple cycle
$cy=``c_1,c_2,c_3,c_4,c_5,c_1"$, a set of vertices
$U=\{u,c_2,c_4\}$, and $I=U\cap cy=\{c_2,c_4\}$, then $u\in U$ and
$u\not \in I$.{\it We use reforming on $cy$ to get a new cycle
including $U$}. $u$ connects with $cy$ via two vertex disjoint
pathes $P_{u,c_2},P_{u,c_5}$ which are terminating at $cy$ at
vertices $c_2,c_5$ respectively. Also name the path
$P=``c_2,c_1,c_5"$ between $c_2,c_5$ along the cycle $cy$, easy to
see $\{P\setminus \{c_1,c_2\}\}\cap U=\emptyset$, then we can reform
$cy$ by replacing $P\setminus \{c_1,c_2\}$ with $P_{u,c_2}\cup
P_{u,c_5}$. After reformation, we have a new simple cycle
$cy'=``u,c_2,c_3,c_4,c_5,u"$ and $cy\cap U\subset cy'\cap U$. We
call this operation as the reforming of $cy$ with respect to $U$.

Sometimes we will say to reform a simple cycle with respect to $U$,
or {\it apply reforming} for abbreviation. Similarly, on
$twin-cycle$, if $c_1,c_2$ are on a half cycle of the $twin-cycle$,
we can define the same operation.

\begin{lemma}\label{no-twin-cycle-in-3-connected}
In a 3-connected graph $G(V,E)$ given a vertex set $U\subseteq V$ and $|U|\geq 4$, if
there is a $twin-cycle$ on $U$, then there is a $K_4$ minor on
$U$.
\end{lemma}
\begin{proof}
Suppose there is $twin-cycle$ $T$ on $U$ with crossing vertices
$\{a,b\}$, and three half cycles are $cy_1,cy_2,cy_3$. For
convenience, set $cy_i'=cy_i\setminus \{a,b\}(i\in \{1,2,3\})$. By
Definition~\ref{twin-cycle}, $U\cap cy'_i\not =\emptyset
(i=\{1,2,3\})$, so assume $u_i\in cy'_i (i=\{1,2,3\})$.
Because $|U|\geq 4$, we have cases as below:
\begin{enumerate}
\item\label{case-4-ver}{there is $cy_i, i\in \{1,2,3\}$ with $|cy_i\cap U|>1$:} I.e. $|U\cap T|\geq 4$.  assume
$i=1$, then in cycle $cy=cy_1\cup cy_2$, there is $|U\cap cy|\geq
3$. Because $G$ is 3-connected, $u_3$ can connect with $cy$ without
passing through $\{a,b\}$. I.e., $u_3$ can connect with $cy$ at
$CR=\{a,b,c\}$. Easy to see, no matter $c\in cy_1$ or $c\in cy_2$,
$CR$ and $\{U\cap cy\}\setminus CR$ can not be two clusters of any
series of $cy$. By Lemma~\ref{one-vertex-plus-one-cycle}, there is a
$K_4$ minor on $U$.

\item\label{case-3-ver}{$|cy_i\cap U|=1(i\in \{1,2,3\})$:} I.e. $|U\cap T|=3$.
    Because $|U|\geq 4$, there is $u_4\in U$ and $u_4\not\in T$.
    Because $G$ is $3$-connected, from $u_4$ to $T$ there are three vertex disjoint pathes
    $p_1,p_2,p_3$ crossing with $twin-cycle$ at $\{c_1,c_2,c_3\}$
    respectively.
    \begin{enumerate}
    \item{$c_1,c_2,c_3$ belong to one half cycle:} assume $i=1$.
    And the order along $cy_1$ is $``a,c_1,c_2,c_3,b"$:
        \begin{enumerate}
        \item{$u_1$ is between $c_1,c_3$:}\label{case-constract} by contracting $cy'_2$
        with $a$ to be $u'_2$, in cycle $cy=cy_1\cup cy_3$ there is $|cy\cap
        U|=3$ in which $\{u_1,u'_2,u_3\}$ and $\{c_1,c_2,c_3\}$ are not
        two clusters in any series of $cy$. By Lemma~\ref{one-vertex-plus-one-cycle},
        there is one $k_4$ minor on $U$.

        \item{$u_1$ is not between $c_1,c_3$:} assume the order is $``a,u_1,c_1,c_2,c_3,b"$,
            then we can get one new $twin-cycle$ $T'$ by reforming the $twin-cycle$ with respect to $U$.
            Then we have $|T'\cap U|\geq 4$ and can be analyzed as Case~\ref{case-4-ver}.
        \end{enumerate}
    \item{$c_1,c_2,c_3$ belong to two half cycles:}
         assume $c_1,c_2\in cy_1$ and $c_3\in cy_2$.
         By contracting $cy'_2$ with $a$ or $b$ and Lemma~\ref{one-vertex-plus-one-cycle},
         a $k_4$ minor on $U$ can be constructed as in Case~\ref{case-constract}.
    \item{$c_1,c_2,c_3$ belong to three half cycles:} this Case is
    similar as Case~\ref{case-constract}.
    \end{enumerate}
\end{enumerate}

\qed
\end{proof}

\begin{lemma}\label{get-k-3-3}
In a graph $G(V,E)$ given a vertex $v\in V$, if there is a twin-cycle on $N(v)$ then $G$
has a $K_{3,3}$ minor.
\end{lemma}
\begin{proof}
Assume $\{v_1,v_2,v_3\}\subseteq N(v)$ belong to the three
half-cycles $c_1,c_2,c_3$ of one twin-cycle respectively. And $a,b$
are the crossing points of the twin-cycle. Then on
$\{v,a,b,v_1,v_2,v_3\}$ there is a $K_{3,3}$ minor.

\qed
\end{proof}

Results below are needed in this paper. When we say a
graph is $n$-connected, it means the graph is connected after
removing arbitrary $(n-1)$ vertices.

\begin{theorem}~\cite{Ore1967,Young1971}\label{k33-to-k5}
In a $4$-connected graph $G(V,E)$, if there is a $K_{3,3}$ minor
then there is a $K_5$ minor.
\end{theorem}

\begin{theorem}\label{G-and-G'}~\cite{WeiyaNote1}
Given a chromatic number $5$ graph $G=(V,E)$, if $G$ can not be
reduced by applying minor actions to be a smaller graph with a
chromatic number $\geq 5$, then there is a vertex $v\in V$, such
that $G'=G\setminus v$ is 3-connected and $N(v)$ is a set of kernel
vertices of $G'$.
\end{theorem}

For convenience, call $G,G'$ parent and child graph respectively.
Easy to see $|N(v)|\geq 4$. It has been shown that an extension in
$G'$ can be got by applying minor actions in $G$ which is not
complicated to prove~\cite{WeiyaNote1}.


\begin{conjecture}\label{Ex-conjecture}~\cite{WeiyaNote1}
In a $3$-connected chromatic number $4$ and $K_5$ minor free graph
$G(V,E)$, if $U$ is a set of kernel vertices of $G$, then there is a
$K_4$ minor on $Ex(U)$.
\end{conjecture}

\begin{theorem}\label{equivalence-theorem}~\cite{WeiyaNote1}
In a chromatic number $5$ graph $G$, there exists a $K_5$ minor if
and only if Conjecture~\ref{Ex-conjecture} is correct.
\end{theorem}

\begin{proposition}\label{no-k33}
In a child graph $G'(V',E')$, there is no $K_{3,3}$ minor, otherwise
$G'$'s parent graph $G$ has a $K_5$ minor.
\end{proposition}
\begin{proof}
If there is a $K_{3,3}$ minor in $G'$, then this $K_{3,3}$ minor
exists in $G$. By Theorem~\ref{k33-to-k5}, there is a $K_5$ minor in
$G$.

\qed
\end{proof}

\section{Simple Cycle}\label{simple-cycle}

In this section, we will show some interesting properties of a child
graph $G$.

\begin{lemma}\label{path-must-joint}
In a graph $G(V,E)$ given a vertex set $U\subseteq V$ and $|U|\geq 4$ which
has no $K_4$ minor and contained in a cycle $cy$, if
along the cycle there are $\{u_1,u_2,u_3,u_4\}\subseteq U$, then any
path $P_{1,3}$ between $u_1,u_3$ crosses with any path $p_{2,4}$
between $u_2,u_4$ in graph $G$.
\end{lemma}
\begin{proof}
The proof is straightforward.

\qed
\end{proof}

\begin{lemma}\label{unique-series}
In a graph $G(V,E)$ given a vertex set $U\subseteq V$ and $|U|\geq 4$, if U is contained
in a cycle $cy$ and has no $K_4$ minor, then
isomorphically there is one unique series on $U$.
\end{lemma}
\begin{proof}
At first, we can define one a $s$ of $U$ on $cy$. Suppose $U$ has
another series $s'$ which is different from $s$, isomorphically
$s,s'$ have at least four vertices with different order, assume
which are $\{v_1,v_2,v_3,v_4\}$, and assume in $s$ the order is
$``v_1,v_2,v_3,v_4"$. Because $\{v_1,v_2,v_3,v_4\}$ are on cycle
$cy$, if there is no $K_4$ minor on $U$, by
Lemma~\ref{path-must-joint}, every path $P_{1,3}$ between $v_1,v_3$
crosses with one arbitrary path $P_{2,4}$ between $v_2,v_4$. So in
$s'$, beginning at $v_1$, the order of $\{v_1,v_2,v_3,v_4\}$ can
only be $``v_1,v_2,v_3,v_4''$ or $``v_1,v_4,v_3,v_2''$, which are
isomorphic. And this is a contradiction with assumption.

\qed
\end{proof}

By Lemma~\ref{unique-series}, for a certain set of vertices
$U$, {\it if there is no $K_4$ minor on $U$, we do not distinguish a
cycle $cy$ with $U\subseteq cy$ from the series of $U$ in $G$.}

\begin{theorem}\label{simple-cycle-in-3-connected}
In 3-connected graph $G(V,E)$ given a vertex set $U\subseteq V$ and $|U|\geq 4$, if there
is no $K_4$ minor on $U$ then there is a simple cycle containing
$U$.
\end{theorem}
\begin{proof}
Choose $u_1,u_2\in U$, because $G$ is 3-connected, there is one
cycle $cy$ including $u_1,u_2$, i.e. we can assume $|cy\cap U|\geq
2$. If there is $u\in U$ and $u\not \in cy$, because $G$ is
3-connected, $u$ can connect with $cl$ on a set of vertices $CR$ and
$|CR|\geq 3$. Assume there is $CR=\{c_1,c_2,c_3\}$ and ordered as
$c_1,c_2,c_3$ along $cy$. $u$ can connect with $\{c_1,c_2,c_3\}$ via
disjoint pathes $P_1,P_2,P_3$ respectively. According to $|cy\cap
U|$, we have cases as below:
\begin{enumerate}
\item{$|cy\cap
U|=2$:} If $\{cy\cap U\}$ and $CR$ are two clusters of a series of
$cy$, $cy$ can be reformed to be a new cycle and includes
$\{u_1,u_2,u\}$ simultaneously, then the condition can be analyzed
as the $\geq 3$ case.

\quad If $\{cy\cap U\}$ and $CR$ are not two clusters of any series
of $cy$, then we can assume $u_1\in cy_{c_1,c_2}\setminus
\{c_1,c_2\}$ and $u_2\not \in cy_{c_1,c_2}$. Then there is one
$twin-cycle$ with $c_1,c_2$ as the crossing vertices on
$\{u,u_1,u_2\}$. Then by Lemma~\ref{no-twin-cycle-in-3-connected},
there is one $K_4$ minor on $U$, which is a contradiction.

\item{$\geq 3$:} By Lemma~\ref{one-vertex-plus-one-cycle}, $CR$ and $\{U\cap cy\}$
are two clusters of a series of $cy$. So we can reform $cy$ with
respect to $U$. I.e., one more vertex in $U$ can be included and no
other vertices in $U$ excluded. So iteratively, $cy$ can be reformed
to be a cycle containing all vertices of $U$.
\end{enumerate}

\qed
\end{proof}

\begin{proposition}\label{2-connected-structure-1}
In a 2-connected graph $G(V,E)$ given a vertex set $U\subseteq V$ which has
no $K_4$ but is contained in a cycle $cy$, then for a vertex $u\in U$ and
a vertex $u'\in N(u)\cap U\not =\emptyset$, if $U\cap \{cy_{u,u'}\setminus \{u,u'\}\}\not =\emptyset$ and
$U\cap \{cy_{u',u}\setminus \{u,u'\}\}\not =\emptyset$, then
$\{u,u'\}$ is a cut set of $G$.
\end{proposition}
\begin{proof}
Assume $u_1\in cy_{u',u}, u_2\in cy_{u,u'}$. By
Lemma~\ref{path-must-joint}, if there is no $K_4$ minor on $U$,
every path between $u_1,u_2$ crosses with arbitrary a path between
$u,u'$. Because there is edge $e(u,u')$ which is one path between
$u,u'$, any path can only cross with it via $u$ or $u'$. Hence
$u,u'$ is a cut set of $G$.

\qed
\end{proof}

When $U$ is included by a cycle in $2$-connected graph, Proposition~\ref{2-connected-structure-1}
describes the structure of $U$. Below
Proposition~\ref{3-to-3-connected-structure}
and~\ref{3-to-2-connected-structure} describe structures of $Ex(U)$
when a vertex $u\in U$ is deleted from a $3$-connected graph $G$.

\begin{lemma}\label{constrain-degree}
In a 3-connected graph $G(V,E)$ given a vertex set $U\subseteq V$, if there exists
a vertex $u\in U$ with $|N(u)\cap U|\geq 3$, there is a $K_4$ minor on $U$.
\end{lemma}
\begin{proof}
Suppose there is such a vertex $u$ with $\{u_1,u_2,u_3\}\subseteq
\{N(u)\cap U\}$. By Theorem~\ref{simple-cycle-in-3-connected},
$\{u,u_1,u_2,u_3\}$ are included by a simple cycle. Then by
Proposition~\ref{2-connected-structure-1}, there is a $K_4$ minor on
$U$.

\qed
\end{proof}

\begin{proposition}\label{3-to-3-connected-structure}
In 3-connected graph $G(V,E)$, $U\subseteq V$, $|U|\geq 4$, $u\in
U$, if there is no $K_4$ minor on $Ex(U)$, then in graph
$G'=G\setminus u$, $U'=\{U\setminus u\}\cup N(u)$, if $G'$ is
3-connected, in $G'$ there is no $twin-cycle$ on $U'$; and there is
one simple cycle $cy$ with $U'\subseteq cy$.
\end{proposition}
\begin{proof}
By Lemma~\ref{constrain-degree}, $|N(u)\cap U|\leq 2$, so
$|N(u)\setminus U|\geq 1$. Hence there is $|U'|\geq 4$. By
Definition~\ref{extending}, $U'$ is a extension of $U$. By
assumption, there is no $K_4$ minor on $U'$. If $G'$ is 3-connected,
by Lemma~\ref{no-twin-cycle-in-3-connected}, there is no
$twin-cycle$ on $U'$; by Lemma~\ref{simple-cycle-in-3-connected},
there is one simple cycle $cy$ in $G'$ with $U'\subseteq cy$.

\qed
\end{proof}

\begin{proposition}\label{3-to-2-connected-structure}
In a 3-connected graph $G(V,E)$ given a vertex set $U\subseteq V$ with $|U|\geq
4$ and a vertex $u\in U$, set a graph $G'=G\setminus
u$ and $U'=\{U\setminus u\}\cup N(u)$ when there is no $K_4$ minor on $Ex(U)$.
If $G'$ is 2-connected, then on $U'$ there is a simple cycle $C$ satisfying:
\begin{enumerate}
\item $\{U\setminus u\}\subseteq C$.
\item if $u_1\in U'\setminus U$, then $u_1\in C$ otherwise $u_1$ connects with $C$ at
exact two vertices. The two vertices form a cut set of $G'$ by which $\{u_1\}$ and
$\{U\setminus u\}$ are isolated.
\end{enumerate}
\end{proposition}
\begin{proof}
By Definition~\ref{extending}, $U'$ is a extension of $U$, so there
is no $K_4$ on $U'$. If there is a $twin-cycle$ on $U\setminus u$ in
$G'$, then this $twin-cycle$ exists in $G$ on $U$, and by
Lemma~\ref{no-twin-cycle-in-3-connected}, there is one $K_4$ minor
on $U$ which is a contradiction. Hence there is no $twin-cycle$ on
$U\setminus u$ in $G'$.

Because $G'$ is 2-connected, by using reforming method,
easy to prove
there is one $K_3$ division on $U\setminus u$. Assume the $K_3$
division includes vertices $\{v_1,v_2,v_3\}\subseteq \{U\setminus
u\}$. If the division is not a cycle, then we assume the division
has a cycle $cy$ on $v_1,v_2,v$ and $v_3$ connects with the cycle at
$v$. Because $G'$ is 2-connected, $v_3$ can connect with the cycle
at a different vertex $v'$. If $v'\in cy_{v_1,v}\setminus v$ on the
division, then we can form a cycle along $v_1,v',v_3,v,v_2,v_1$; the
same for $v'\in P_{v,v_2}\setminus v$. If $v'\in
P_{v_1,v_2}\setminus \{v_1,v_2\}$, then treat $v,v'$ as the two
crossing vertices, there is one $twin-cycle$ with three half-cycles
on $\{v_1,v_2,v_3\}$ respectively, which is a contradiction with no
$twin-cycle$ on $U\setminus u$. Hence we can assume there is a cycle
on $\{v_1,v_2,v_3\}$. Similar as proof of
Lemma~\ref{simple-cycle-in-3-connected}, we can reform to get a
cycle $C$ including all vertices of $U\setminus u$.

If there is $u_1\in U'\setminus U$, because $G'$ is 2-connected,
there are two cases:
\begin{enumerate}
\item{$u_1$ connects with $C$ at $\geq 3$ vertices:} assume at
$W$. Then by Lemma~\ref{one-vertex-plus-one-cycle}, $W$ and
$\{U'\cap C\}$ are two clusters. Then $C$ can be reformed to contain
$u_1$ and do not exclude any vertex in $U'\cap C$ out.
\item{$u_1$ connects with $C$ at $2$ vertices:} assume at
$W=\{c_1,c_2\}$. Then $\{c_1,c_2\}$ is a cut-set of $G'$ and in
graph $G'\setminus \{c_1,c_2\}$, $u_1$ and $U\setminus u$ are in
different components.
\end{enumerate}

\qed
\end{proof}

Proposition~\ref{distribution-of-k-4-in-2-connected} show while
keeping $K_4$ minor, a $2$-connected graph can be reduced by
applying minor actions.

\begin{proposition}\label{distribution-of-k-4-in-2-connected}
In a 2-connected graph $G(V,E)$, let $\{v_1,v_2\}$ be a cut-set and in
$G'=G\setminus \{v_1,v_2\}$ $C_1,C_2,...,C_x$ be components. If
$G$ has a $K_4$ minor, then there is a $K_4$ minor on $C_i\cup
\{v_1,v_2\}, i\in \{1,2,...,x\}$.
\end{proposition}
\begin{proof}
Because $\{v_1,v_2\}$ is cut-set, every $K_4$ minor is on at most
two components. Assume there is a $K_4$ minor on vertices
$\{c_1,c_2,c_3,c_4\}\subseteq C_1\cup C_2\cup \{v_1,v_2\}$, then the
only possibility is $c_1\in C_1$ and $\{c_2,c_3,c_4\}\subseteq
C_2\cup \{v_1,v_2\}$, and $\{v_1,v_2\}\setminus \{c_2,c_3,c_4\}\not
=\emptyset$. In order to get a $K_4$ minor, $c_1$ must be contracted
with $v_1$ or $v_2$, hence a $K_4$ minor on $C_2\cup \{v_1,v_2\}$
can be constructed.

\qed
\end{proof}


\section{Wagner's Equivalence Theorem}\label{wagner}

The results in section~\ref{simple-cycle} can be applied in
induction method to prove more interesting and useful results.
Next as an exercise we will show how to use this induction to prove
Wagner's Equivalence Theorem. Note that our proof does
not depend on Kuratowski's Theorem.

\begin{definition}\label{formal-graph}
In a graph $G(V,E)$ given a non-empty vertex set $U\subseteq V$, we call
$G$ is the $U$'s formal graph, if
there is no $K_5$ or $K_{3,3}$ minor after adding a vertex $v$ to $G$ with
$N(v)=U$.
\end{definition}

\begin{lemma}\label{keeping-formal}
If a graph $G(V,E)$ is the formal graph of $U\not=\emptyset$, given a vertex $u\in U$
and $U'$ the extension of $U$ in graph $G'=G\setminus u$,
then $G'$ is a formal graph of $U'$.
\end{lemma}
\begin{proof}
The non-trivial condition is $U'=\{U\setminus u\}\cup N(u)$. Suppose
after adding $v'$ to $G'$ to get a new graph $G'_v$, there is $K_5$
or $K_{3,3}$ minor. Then in graph $G$, if we add $v$ with $N(v)=U$,
then contract $v$ with $u$ to be vertex $v'$, the new graph is
$G'_v$, hence there is $K_5$ or $K_{3,3}$ minor which is a
contradiction with assumption.

\qed
\end{proof}

\begin{lemma}\label{twin-cycle-2-k33}
In a graph $G(V,E)$ given a vertex set $U\subseteq V$, if $G$ is the $U$'s formal graph
then there is no $twin-cycle$ on $U$.
\end{lemma}
\begin{proof}
If there is a $twin-cycle$ $T$ on $U$, after adding $v$ to $G$ with
$N(v)=U$, by Lemma~\ref{get-k-3-3}, there is a $K_{3,3}$ minor,
which is a contradiction with Definition~\ref{formal-graph}.

\qed
\end{proof}

\begin{theorem}\label{U-on-perimeter}
In a connected graph $G(V,E)$ given a vertex set $U\subseteq V$, if $G$ is $U$'s formal
graph,
then $G$ is planar and $U$ is contained by $G$'s boundary of an external face. 
\end{theorem}
\begin{proof}
We prove this by induction on $|V|$. When $|V|\leq 3$, easy to
verify the conclusion holds.

When $|V|>3$, there are three cases:
\begin{enumerate}
\item {$G$ is $3$-connected:} $u\in U$, set $G'(V',E')=G\setminus u$,
$U'=\{U\setminus u\}\cup N(u)\not= \emptyset$ is an extension of
$U$. By Lemma~\ref{keeping-formal}, $G'$ is a formal graph of $U'$.
By induction, $G'$ is planar, and $U'$ is on boundary of an external
face of $G'$. If restore $u$, $G$ is still planar. And $N(u)$ and
$U\setminus N(u)$ are two clusters, otherwise there is one
$twin-cycle$ on $U$ in $G$, which is a contradiction with
Lemma~\ref{twin-cycle-2-k33}. Hence, in $G$, $U$ is on boundary of
an external face.

\item {$G$ is $2$-connected:} Assume $W=\{w_1,w_2\}$ is an arbitrary cut-set of $G$, and $C_1$ is
a component in graph $G\setminus W$. In graph $G_1=C_1\cup W$, we
can assume there is edge $e(w_1,w_2)$, because otherwise we can
apply contraction(minor action) on $G\setminus \{C_1\cup W\}$ to
contract $W$ to be a vertex, and the proof is similar. Then graph
$G_1$ has $U_1= \{U\cap C_1\}\cup \{W\}$ as an extension of $U$ by
Definition~\ref{extending}. Similar as proof of
Lemma~\ref{keeping-formal}, $G_1$ is a formal graph of $U_1$.

\quad By induction $G_1$ is planar and $U_1$ is on boundary of an external face. If
in graph $G\setminus W$, there is another component $C_2$ and set
$G_2=C_2\cup W$, then $G_1,G_2$ can be combined together by merging
$W$, and after combination, it is planar and $U_1\cup U_2$ is on boundary of an external face.
If in graph $G\setminus W$, there are only components
$C_1,C_2$, then this subcase has been proved. 

\quad If in graph $G\setminus W$ besides $C_1,C_2$, there is another
component $C_3$, then in $G$ there is one $twin-cycle$ on $U$, which
is a contradiction with Lemma~\ref{twin-cycle-2-k33}. Hence we can
conclude that when $G$ is $2$-connected, the conclusion holds.

\item {$G$ is $1$-connected:} This can be proved similarly as $2$-connected case.
\end{enumerate}

\qed
\end{proof}

Theorem~\ref{U-on-perimeter} has closed connection with
Theorem~\ref{simple-cycle-in-3-connected},
Proposition~\ref{3-to-3-connected-structure}
and~\ref{3-to-2-connected-structure}. Easy to see if condition of
$K_{3,3}$ minor added, from results of
Theorem~\ref{simple-cycle-in-3-connected},
Proposition~\ref{3-to-3-connected-structure}
and~\ref{3-to-2-connected-structure}, we can prove
Theorem~\ref{U-on-perimeter} easily.

From Lemma~\ref{keeping-formal} and Theorem~\ref{U-on-perimeter}, by
intuition in a planar graph, the external face can be peeled
iteratively. A reverse processing can be used to generated a planar
graph. From these ideas, we can give a new geometrical definition of
planar graph. With such a definition, simple algorithms can be
designed to test planarity of graph and compute an orthogonal planar
embedding of planar graph. Because of limit space, all of these will
be discussed in another paper.


\begin{lemma}\label{parent-graph}
If $G(V,E)$ is a 4-connected chromatic number $5$ and $K_5$ minor free
graph, then graph $G$ is planar.
\end{lemma}
\begin{proof}
There is graph $G'=G\setminus v$ where $v\in V$. By
Theorem~\ref{k33-to-k5} and Proposition~\ref{no-k33}, $G'$ is a
formal graph of $N(v)$. By Theorem~\ref{U-on-perimeter}, $G'$ is
planar, and $N(v)$ is on boundary of an external face of $G'$. Hence
$G$ is planar.

\qed
\end{proof}


\begin{lemma}\label{4cc}
If each planar graph is $4$-colorable, then
Conjecture~\ref{Ex-conjecture} is correct.
\end{lemma}
\begin{proof}
We prove this Lemma by contradiction. Suppose
Conjecture~\ref{Ex-conjecture} is not correct, then by
Theorem~\ref{equivalence-theorem}, there is a chromatic number $5$
graph $G$ in which there is no $K_5$ minor. By
Theorem~\ref{G-and-G'}, $G$ is $4$-connected. By
Lemma~\ref{parent-graph}, $G$ is planar. By assumption, $G$ can be
colored with $4$ colors which is a contradiction with $G$ has
chromatic number $5$.

\qed
\end{proof}

\begin{theorem}[Wagner's Equivalence Theorem]\label{wagner-theorem}
Every chromatic number $5$ graph has a $K_5$ minor if and only if
every planar graph can be colored with $4$ colors.
\end{theorem}
\begin{proof}
At first we prove from left to right. If a graph has a $K_5$ minor,
trivially it is not planar. Suppose graph $G(V,E)$ is a planar
graph, so $G$ has no $K_5$ minor, by left side, $G$ has chromatic
number $<5$. Hence $G$ can be colored with $4$ colors.

Then we prove from right to left. If every planar graph is
$4$-colorable, by Lemma~\ref{4cc}, Conjecture~\ref{Ex-conjecture} is
correct. By Theorem~\ref{equivalence-theorem}, every
chromatic number $5$ graph has a $K_5$ minor.

\qed
\end{proof}

\section{Conclusion and Next Step of Work}

In this paper, we prove that in a $3$-connected chromatic number
$4$ graph $G(V,E)$ given a vertex set $U\subseteq V$, if there is no $K_4$ minor
on $U$ then $U$ is included by a simple cycle of $G$. And when
$G$ is a planar graph, the simple cycle is boundary of an external
face. By applying such results, we can prove Wagner's Equivalence
Theorem without using Kuratowski's Theorem which is different from
existing proofs. That means our proof does not rely on current existing
properties of planar graph.

In fact starting from this paper a new geometrical definition of
planar graph can be deduced, so is an algorithm for testing
planarity and computing an orthogonal planar embedding of a planar
graph whose complexity is the same as current algorithms but much
simpler. In next step we will prove the new definition is equivalent
with Kuratowski's Theorem.



\title{Coloring on Outerplanar Graph and Planar Graph}

\author{Weiya Yue\inst{1}, Weiwei Cao\inst{2}}
\institute{$^1$ Google Inc., Mountain View, CA 94043\\
$^2$ Graduate School of Chinese Academy of Sciences, Beijing China 100049\\
\email{weiyayue@hotmail.com}}

\maketitle

\vspace{-5mm}
\section{Abstract}
In this paper, at first we discuss the relation between outerplanar graph and
book thickness 1 graph, mainly based on the perimeter trace. By doing
research on book thickness 1 graph with the help of perimeter trace,
we can prove there is a $3$-color assignment of outerplanar graph with
special properties by induction. More important is, the proof does not depend on
the minimum degree of outerplanar graph and can be generalized. In order to show this, we discuss a
way to extend our results to color a planar graph with $4$ colors.
Beyond that, our conclusion shows that a outerplanar graph and planar
graph can not only be colored with $3$ and $4$ colors respectivly, but also satisfying extra constraints.


\section{Introduction}
An outerplanar graph~\cite{CH1967} is an undirected graph that can
be drawn without edge crossing and whose vertices are on boundary of
the drawing's unbounded or outer face. A graph is outerplanar if
it is turned into a planar graph after adding a new vertex which connect all vertices in the
graph. One method to recognize an outerplanar graph
is to use its criterion: a graph is
outerplanar if and only if it does not contain $K_4$ or $K_{2,3}$
minor~\cite{DR2000}. The
decomposition method to test if every biconnected component is
outerplanar~\cite{SM1979} can also be used. It has been proven that
the book thickness of a graph is $\leq 1$ if and only if it is a outerplanar graph~\cite{BK1979}.
For convenience, below we will call a book thickness 1 graph as book-1 graph.

All the knowledge in literatures on coloring an outerplanar graph is
that it can be colored with $3$ colors since its
minimum vertex degree is no bigger than $2$~\cite{PS1986}. A simple iterative
algorithm can output a $3$-color assignment which removes a degree $\leq 2$ vertex,
then colors the remaining graph, at last restores the removed vertex with the
unused color different from colors assigned to its neighbors.

The concise outlook of book-1 graph makes some properties of
outerplanar graph can be easier found. For example, every book-1
graph, i.e. every outerplanar graph, has a degree $\leq 2$ vertex.
Book-1 graph pictures a better geometric view to make observation
for outerplanar graph, which motivates our work on new results of
coloring of outerplanar graph.

Our paper is organized as follows. In Section~\ref{sec-book-1} we
prepare some definitions and prove the properties of perimeter trace
in outerplanar graph and in book-1 graph.
Section~\ref{sec-regularity} presents more properties of color
assignment of outerplanar graph which does not depend on minimum
degree of outerplanar graph. Then we discusses a way to
generalize our results of $3$-coloring of outerplanar graph to prove
four color problem in planar graph. Section~\ref{sec-conclusion}
concludes.

\section{Perimeter Trace}\label{sec-book-1}

Corresponding to a drawing of a book-1 graph, all vertices are on
the spine. The two vertices at the ends of spine are called
end-vertices. In book-1 graph, vertices expose to the side of the
line where edges are drawn are called outer-vertices; Other vertices
are called inner-vertices. Trivially, end-vertices are always
out-vertices.

Figure~\ref{fig-1}.$a$ displays a book-1 graph. All vertices are on
the spine, edges are at right side of the line. $\{v_1,v_6\}$ are
end-vertices, also the only two outer-vertices.
Figure~\ref{fig-1}.$b$ displays another book-1 graph, in which
$\{v_1,v_6\}$ are end-vertices, $\{v_1,v_3,v_6\}$ outer-vertices,
and $\{v_2,v_4,v_5\}$ inner-vertices.

\begin{figure}
 \setlength{\unitlength}{0.5mm}
\begin{pspicture}(-2,-0.5)(10,3.0)
   \pscurve[linewidth=1pt,%
     showpoints=true](1.5,0.5)(1.5,1)(1.5,1.5)(1.5,2)(1.5,2.5)(1.5,3)
   \pscurve[linewidth=1pt,%
     showpoints=false](1.5,3)(2.1,2.3)(2.1,1.2)(1.5,0.5)
   \pscurve[linewidth=1pt,%
     showpoints=false](1.5,0.5)(1.9,1.25)(1.5,2)
   \pscurve[linewidth=1pt,%
     showpoints=false](1.5,2)(1.7,2.5)(1.5,3)
   \pscurve[linewidth=1pt,%
     showpoints=false](1.5,0.5)(1.7,1)(1.5,1.5)
%
%
\rput(1.1,0.5){$v_6$} \rput(1.1,1){$v_5$}\rput(1.1,1.5){$v_4$}
\rput(1.1,2){$v_3$}\rput(1.1,2.5){$v_2$}\rput(1.1,3){$v_1$}
\rput(1.5,-0.5){$(a)$}
   \pscurve[linewidth=1pt,%
     showpoints=true](6.5,0.5)(6.5,1)(6.5,1.5)(6.5,2)(6.5,2.5)(6.5,3)
   \pscurve[linewidth=1pt,%
     showpoints=false](6.5,0.5)(6.9,1.25)(6.5,2)
   \pscurve[linewidth=1pt,%
     showpoints=false](6.5,2)(6.7,2.5)(6.5,3)
   \pscurve[linewidth=1pt,%
     showpoints=false](6.5,0.5)(6.7,1)(6.5,1.5)
%
\rput(6.0,0.5){$v_6$} \rput(6,1){$v_5$}\rput(6,1.5){$v_4$}
\rput(6,2){$v_3$}\rput(6,2.5){$v_2$}\rput(6,3){$v_1$}
\rput(6.5,-0.5){$(b)$}
 \end{pspicture}
\caption{Book Thickness $1$ Graphs}\label{fig-1}
\end{figure}
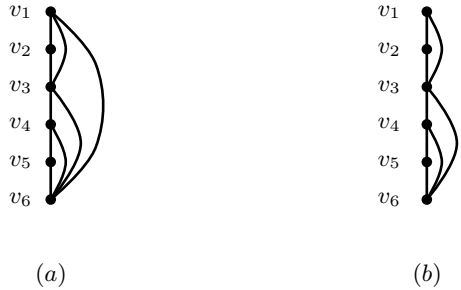

Given a graph $G(V,E)$, a set of vertices $U\subseteq V$ and
$|U|=x>0$, if there is a $K_x$ minor in which any vertex $v\in K_x$
is contracted from $S\subseteq V$ and $S\cap U\neq \emptyset$, then
we say in $G$ the $K_x$ minor is on $U$. If $U'\subseteq V$ and
$U\subseteq U'$, we also say there is a $K_x$ minor on $U'$.

In a $K_{x,y}$ minor, we call the $x$ vertices as upper-vertices,
and the $y$ vertices lower-vertices. A certain set of vertices is
defined as a perimeter trace as below.
\begin{definition}\label{perimeter-trace-of-outplanar}
Given an outerplanar graph $G(V,E)$, its perimeter trace is a set of
vertices $U\subseteq V$ if:
\begin{enumerate}
\item on $U$ there is no $K_3$ minor;
\item no $K_{2,2}$ minor in which all upper-vertices or all lower-vertices belong to $U$.
\end{enumerate}
\end{definition}
It is trivial that all subsets of a perimeter trace $U$ of an
outerplanar graph $G$ are perimeter traces; if we add a vertex $v$
into $G$ and $v$ is only connected with vertices in $U$, then the
new graph is still outerplanar.

In subsection~\ref{subsec-structre-book-1}, we prove some properties
of perimeter trace in outerplanar graph and in book-1 graph.

\subsection{Perimeter Trace in Outerplanar and Book-1 Graph}\label{subsec-structre-book-1}

Note that a cut set containing more than one vertex implies none of
these vertices is a cut vertex. Below we give some preliminary
results.
\begin{lemma}\label{k-2-2-in-outerplanar}
In an outerplanar graph, if there is a $K_{2,2}$ minor on $4$ vertices,
then its upper-vertices or lower-vertices form a cut set.
\end{lemma}
\begin{proof}
Assume there is a $K_{2,2}$ minor on $U=\{u_1,u_2,u_3,u_4\}$, and no
harm to say $\{u_1,u_2\}$ are upper-vertices and $\{u_3,u_4\}$
lower-vertices. If $\{u_1,u_2\}$ is not a cut set, then there is a
path $P_{3,4}$ between $u_3$ and $u_4$ without passing through
$\{u_1,u_2\}$. So $P_{3,4}$ has two cases:
\begin{enumerate}
\item {$P_{3,4}=e(u_3,u_4)$:} if $\{u_3,u_4\}$ is not a cut set, then there
is a path $P_{1,2}$ between $u_1$ and $u_2$ without passing through
$\{u_3,u_4\}$. So on $U$ there is a $K_4$ minor which contradicts
with criterion of outerplanar graph.
\item {$P_{3,4}\not =e(u_3,u_4)$:} assume there is $u\not\in U$ on path $P_{3,4}$, then on
$U\cup \{u\}$ there is a $K_{2,3}$ minor which contradicts with
criterion of outerplanar graph.
\end{enumerate}
\qed
\end{proof}

Given a connected outerplanar graph $G(V,E)$, we can always find two
vertices $u_1,u_2$ between which there is a path
$P_{1,2}=e(u_1,u_2)$ and $\{u_1,u_2\}$ is not a cut set or
$P_{1,2}\neq e(u_1,u_2)$ and every vertex in
$P_{1,2}\setminus\{u_1,u_2\}$ is a cut vertex separating $u_1$ and
$u_2$. For example, two neighbor vertices on a connected outerplanar
graph's outer face can be chosen as $\{u_1,u_2\}$. We name vertices
in $P_{1,2}\setminus \{u_1,u_2\}$ as $\emph{bridge vertices}$ of
$\{u_1,u_2\}$. Below we use $U=\{u_1,w_1,...,w_x,u_2\}$ to represent
the vertices orderly appearing on a path in an outerplanar graph.

\begin{proposition}\label{structure-of-outplanr-perimeter-trace}
In a connected outerplanar graph $G(V,E)$, a set of vertices
$U=\{u_1,w_1,...,w_x,u_2\}\subseteq V$ is a perimeter trace if and
only if there are $W\subseteq V$ as $\{u_1,u_2\}$'s bridge vertices
and $\{w_1,...,w_x\}\subseteq W$.
\end{proposition}
\begin{proof}
If there is no $W\subseteq V$ as $\{u_1,u_2\}$'s bridge vertices,
then there is a cut set $C$ with $|C|\geq 2$ separating $u_1$ and
$u_2$. Then on $C\cup \{u_1,u_2\}$ there is a $K_{2,2}$ minor whose
upper-vertices(lower-vertices) are $\{u_1,u_2\}$, so $U$ is not a
perimeter trace. If $W$ exists, 
and there is $w\in \{w_1,...,w_x\}\setminus W$, and $w$ is not a
cut-vertex separating $u_1$ and $u_2$, then on $\{u_1,u_2,w\}$, i.e. on $U$ there is a $K_3$
minor, which means $U$ is not a perimeter trace.

If $W$ are $\{u_1,u_2\}$'s bridge vertices,
there is no $K_3$ minor on $W\cup \{u_1,u_2\}$; if no two vertices
in $W\cup \{u_1,u_2\}$ form a cut set, then by
Lemma~\ref{k-2-2-in-outerplanar} there is no $K_{2,2}$ minor whose
upper-vertices(lower-vertices) are only contained in $W\cup
\{u_1,u_2\}$. Hence, $W\cup \{u_1,u_2\}$ is a perimeter trace. So
$U\subseteq W\cup \{u_1,u_2\}$ is a perimeter trace.

\qed
\end{proof}

\begin{corollary}\label{book-1-has-perimeter-trace}
In a book-1 graph, its outer-vertices form a perimeter trace.
\end{corollary}
\begin{proof}
Since the union of perimeter traces for each component of a
disconnected outerplanar graph is a perimeter trace of the whole
graph, then it concludes by applying
Proposition~\ref{structure-of-outplanr-perimeter-trace} on the
simple structure of book-1 graph.

\qed
\end{proof}

Also by Proposition~\ref{structure-of-outplanr-perimeter-trace}, if
an outerplanar graph is $2$-connected, a perimeter trace can be
formed by two arbitrary neighbor vertices on the hamiltonian cycle.
In this paper, a $n$-connected graph means that the graph remains
connected after removing arbitrary $(n-1)$ vertices.

\begin{theorem}\label{keep-structure-after-reducing}
Given an outerplanar graph $G(V,E)$, if $U\subseteq V$ is a
perimeter trace, and $u\in U$, then $U'=\{\{U\setminus u\}\cup
N(u)\}$ in graph $G'=G\setminus u$ is a perimeter trace.
\end{theorem}
\begin{proof}
There are two cases depending on $G$ is connected or disconnected.
When $G$ is disconnected, note that the union of perimeter traces of
components of $G$ is $G$'s perimeter trace, which means we can do
the proof on the component containing $u$. Hence below we only prove
the case when $G$ is connected. Assume $U=\{u_1,u_2,...,u_x\}$
orderly appearing on a path in $G$.

When $G$ is connected, there are two cases: $G'$ is connected or
disconnected. By
Proposition~\ref{structure-of-outplanr-perimeter-trace}, if
$G'=G\setminus u$ is connected, then $u\in \{u_1,u_2\}$, we can assume $u=u_1$ and get two
subcases:
\begin{enumerate}
\item $deg(u)=1$, trivially $U'$ is a perimeter trace of $G'$.
\item $deg(u)>1$, assume $\{q_1,q_2,...,q_y\} \subseteq N(u)$ orderly appearing on a path $P_{1,2}$ from $q_1$ to $u_2$.
Assume can not find a path $P_{1,2}$ between $q_1$ and $u_2$,
$P_{1,2}\setminus \{q_1,u_2\}$ are all bridge vertices of $\{q_1,u_2\}$,
then there is a cut set $C$ with $|C|\geq 2$ separating $q_1$ and
$u_2$. Because $U$ is a perimeter trace on $G$, by
Proposition~\ref{structure-of-outplanr-perimeter-trace}, there is a
path $P$ between $u_1,u_2$, and $P'=P\setminus \{u_1,u_2\}$ are
bridge vertices of $\{u_1,u_2\}$. Then there is $C\cap P'=\emptyset$. So
on $C\cup \{u_1,q_1,u_2\}$ there is a $K_{2,3}$ minor which is a
contradiction. Hence on path $P_{1,2}$, $P_{1,2}\setminus \{q_1,u_2\}$ are all bridge vertices.
Also similarly to show that, $N(v) \subseteq P_{1,2}$.
By
Proposition~\ref{structure-of-outplanr-perimeter-trace}, $U'$ is a
perimeter trace of $G'$.
\end{enumerate}

If $G'=G\setminus u$ is disconnected, then $u$ is a cut vertex of
$G$. Suppose $V=\{v_1,v_2,...u,...,v_n\}$ orderly appear on the
spine. This case can be proved similarly by doing induction
on two subgraphs of $G$: $G_1=G\cap \{v_1,...,u\}$ and $G_2=G\cap
\{u,...,v_n\}$. Only need to notice that after deleting $u$ from
$G_1,G_2$ respectively, the union of two perimeter traces of
$G_1,G_2$ is a perimeter trace of $G'$.
\qed
\end{proof}

In a book-1 graph, by Corollary~\ref{book-1-has-perimeter-trace},
its outer-vertices form a perimeter trace. And the conclusion in
Theorem~\ref{keep-structure-after-reducing} can be applied on such a
perimeter trace easily because of simple structure. But
Theorem~\ref{keep-structure-after-reducing} is more generalizable.
It states every perimeter trace satisfies such a conclusion. And
in~\cite{WeiyaNote2}, similar conclusion in planar graph can be
proved as Theorem~\ref{keep-structure-after-reducing}.


\begin{theorem}\label{book-1}
Given an outerplanar graph $G(V,E)$ and $U=\{u_1,w_2,...,w_x,u_2\}$
a perimeter trace in which $W=\{w_2,...,w_x\}$ are bridge vertices of $\{u_1,u_2\}$,
$G$ can be transfer to a book-1 graph in a way
such that $u_1$ and $u_2$ are end-vertices, $U$ are outer-vertices,
and the order of $U$ on the spine is kept the same as they
are on the outer face of $G$.
\end{theorem}
\begin{proof}
We make induction on $|V|$. If $|V|\leq 2$, the conclusion holds
trivially. When $|V|=n$, set $G'=G\setminus u_1$. By
Theorem~\ref{keep-structure-after-reducing}, $Ex(U)=\{U\setminus
u_1\}\cup N(u_1)$ are contained in a perimeter trace of $G'$. Assume
$U'=\{u_1',...,u_y',w_2,...,w_x,u_2\}$ is a perimeter trace of $G'$
and $Ex(U)\subseteq U'$, then $N(u_1)\subseteq
\{u_1',...,u_y',w_2\}$.

By induction, $G'$ can be transferred to a book-1 graph $S'$ in
which $U'$ are outer-vertices, $\{u_1',u_2\}$ are the two
end-vertices, and the order of $U'$ is kept as that on outer face of
$G'$. By definition of book-1 graph, we can get a graph $S$ after
adding $u_1$ into $S'$ by putting $u_1$ at $u_1'$ end. Because
$N(u_1)\subseteq \{u_1',...,u_y',w_2\}$ and $\{u_1',...,u_y',w_2\}$
are outer-vertices of $S'$ and edges between $u_1$ and $N(u_1)$ have
no crossing with edges of $S'$, $S$ is a book-1 graph. Moreover,
such edges should not cover $w_2$, hence $U$ are the outer-vertices with $U$'s order as in
$G$'s outer face is kept.
\qed
\end{proof}

Theorem~\ref{book-1} shows that for every perimeter trace of an
ourerplanar graph, there is a book-1 graph whose outer-vertices
include the perimeter trace.

In next section, we will show a
proof of $3$-coloring of
outerplanar graph. The most important reason to do so is that the
new proof does not need using minimum degree property which makes it
more generalizable to prove four color theorem of planar graph, so is the Hadwiger Conjecture.

\section{Generalizable Coloring of book-1 Graph}\label{sec-regularity}

In this Section, we prove a book-1 graph, i.e. an outerplanar graph,
has a $3$-color assignment satisfying certain coloring constraints.
More important, the proof does not depend on minimum degree property
of outerplanar graph and we find it can be generalized the
results to coloring a planar graph. At first we define $series$ and
$cluster$ introduced in~\cite{WeiyaNote2}. A sequence of vertices
appearing along a simple path in one direction form a $seires$.
After deleting vertices in a $series$ of a path, the left $series$
is still called a $series$ of the path. A continuous part of a
$series$ is called a $cluster$.

If we decompose a $series$ $S$ into a sequence of clusters
$\Upsilon=\{\Upsilon_1,\Upsilon_2,...,\Upsilon_x\}$ by order and
$\bigcup_{i=1}^{x} \Upsilon_i = S$, $\Upsilon$ is called clusters of
the series $S$, in which the end of $\Upsilon_i(1\leq i<x)$ only
overlap on $\leq 1$ vertex with the beginning of $\Upsilon_{i+1}$.
So a vertex can belong to more than one cluster. Definitions
of $series$ and $cluster$ can be generalized on a set of vertices
$U=\{u_1,u_2,...,u_x\}$ if $u_i,u_{i+1}$ are connected and there is
a path $P_i$ between $u_i,u_{i+1}$ satisfying $P_i\cap
U=\{u_i,u_{i+1}\}$ for $1\leq i<x$. If $u_i,u_{i+1}(1\leq i<x)$ are
disconnected we assume an edge $e(u_i,u_{i+1})$ which is certainly a
path $P_i$.

Assume vertices of a book-1 graph $G$ orderly appear on the spine as $V=\{v_1,v_2,...,v_n\}$ and clusters on ordered set of $G$'s
outer-vertices
are $\Upsilon=\{\Upsilon_1,\Upsilon_2,...,\Upsilon_x\}$.
Name $\Upsilon_i$ and $ \Upsilon_{i+1}(1\leq i< q)$ two neighbor
clusters.

To a $cluster$ $\Upsilon_i$, we add a vertex $\gamma_i$ with
$N(\gamma_i)=\Upsilon_i$, and name $\gamma$ as $\Upsilon_i$'s
cluster-vertex. In Figure~\ref{fig-2}.$a$, $\{\Upsilon_1,...,\Upsilon_6\}$ are clusters defined on vertex $\{v\}$ whose cluster-vertices are
corresponding to $\{\gamma_1,...,\gamma_6\}$. In Figure~\ref{fig-2}.$b$, $\{v_1,v_3,v_6\}$ is a series of
outer-vertices which is decomposed into clusters $\{\Upsilon_1,...,\Upsilon_8\}$. Specifically,
$\Upsilon_1=\Upsilon_2=\{v_1\}$, $\Upsilon_3=\{v_1,v_3\}$, $\Upsilon_4=\{v_3\}$, $\Upsilon_5=\{v_3,v_6\}$,
$\Upsilon_6=\Upsilon_7=\Upsilon_8=\{v_6\}$. Here we give the definition:

\begin{definition}\label{color-collection}
Given $CL=\{1,2,...,n\}$ colors, a color collection $cn$ is a real subset of $CL$ and $|cn| = n-1$. A subset $cc \subset cl$
whose cardinality is $< n-1$ can be expanded to be a set of color collections by adding extra colors. Two set of color collections
$cc_1, cc_2$, we say $cc_1$ and $cc_2$ are consistent if $cc_1\cap cc_2 \neq \emptyset$; $cc_1$ and $cc_2$ are inconsistent if
$cc_1\neq cc_2$.
\end{definition}

Without confusion, we do not distinguish a color subset and the set of color collections expanded from it.

Below assume colors $CL=\{1,2,3\}$ are used, and a color collection is
$\{1,2\}$,$\{1,3\}$, or $\{2,3\}$. Color $\{1\}$
is equivalent with collections $\{\{1,2\},\{1,3\}\}$, similarly
color $\{2\} = \{\{1,2\},\{2,3\}\}$ and color $\{3\}=\{\{1,3\},\{2,3\}\}$. When only $3$ colors are used, a $cluster$ $\Upsilon_i$ is colored with the collection $cn$, if and only if its
cluster-vertex $\gamma_i$ is colored with $CL\setminus cn$. For
example, $color(\Upsilon_i)=\{1,2\}$ is equivalent with
$color(\gamma_i)=\{3\}$. Below $cn(\Upsilon_i)$ represents the
collection of colors used on $\Upsilon_i$.
To give some examples about consistency and inconsistency, $\{\{1,2\}\}$ and $\{\{1,2\},\{1,3\}\}$ are consistent and also inconsistent.
So is $\{\{1,2\}\}$ and $\{1\}$. $\{1,2\}$ and $\{1,3\}$ are inconsistent, and $\{\{1,2\},\{1,3\}\}$
and $\{\{2,3\}\}$ are inconsistent, so is $\{1\}$ and $\{\{2,3\}\}$.

\begin{definition}\label{collection-constraints}
In a book-1 graph $G(V,E)$, $U$ are outer vertices, rules of
collection-constraints of coloring on $U$ are defined following:
\begin{enumerate}
\item a $cluster$ is colored with colors from a collection;
\item two neighbor clusters have the same or different collections.
\end{enumerate}
\end{definition}

\begin{definition}\label{extend-original-graph}
In a book-1 graph $G$, collection-constraints $ct$ are defined on
its outer-vertices, define constraint-graph $G_{ct}$ by extending
$G$ as following: 1)every $cluster$ has a cluster-vertex, and if
$cn(\Upsilon_i)=cn(\Upsilon_j)$ the two clusters share a
cluster-vertex; 2)if $cn(\Upsilon_i)\not=cn(\Upsilon_j)$, there is
edge $e(\gamma_i,\gamma_j)$.

\end{definition}

Cluster-vertices($\gamma$-vertices) can have an order
according to that of their belonging clusters.
In Figure~\ref{fig-2}.$a$, the constraint-graph by constraints $cn(\Upsilon_1)\not=cn(\Upsilon_2)$ and
$cn(\Upsilon_2)\not=cn(\Upsilon_3)$ is displayed. If a more constraint $cn(\Upsilon_5)=cn(\Upsilon_6)$ is added, vertices
$\{\gamma_5,\gamma_6\}$ will be merged into a new vertex $\gamma_{5,6}$. Figure~\ref{fig-2}.$b$ is another similar example of
constraint-graph, and its corresponding constraints can be found easily.

\begin{observation}\label{structure-of-constraint-graph}
$G_{ct}\setminus G$ is a $K_2$ or $K_1$ subdivision, and $G_{ct}$ is a
book-1 graph.
\end{observation}
\begin{proof}
Obviously $G_{ct}\setminus G$ can only be a $K_2$ or $K_1$ subdivision.
Turn around the $K_2$ or $K_1$ subdivision to an end of $G$ in $G_{ct}$
then the resulted graph is a book-1 graph.
\qed
\end{proof}

\begin{observation}\label{trivial-satisfy-something}
There is a color assignment of $G_{ct}$ uses $\leq 3$ colors,
if and only if there is a color assignment $cl$ of $G$ using $\leq 3$ colors and satisfying
collection-constraints.
\end{observation}
\begin{proof}
$cl_{ct}$ can be used on $G$ to be $cl$. $cl$ satisfies all
collection-constraints following definitions of
collection-constraints and constraint-graph.
\qed
\end{proof}
{\em Since $G_{ct}$ is a book-1 graph, which is known to be $3$-colorable, there always exists such a
$cl_{ct}$}.
In Figure~\ref{fig-2}, there are some examples of $G_{ct}$.

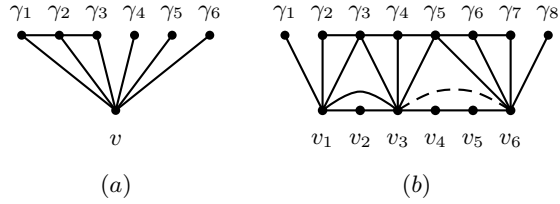
\begin{figure}
 \setlength{\unitlength}{0.5mm}
\begin{pspicture}(-2,0)(10,3.0)
    \pscircle*(1.5,1){0.06}
    \pscircle*(0.25,2){0.06}
    \pscircle*(0.75,2){0.06}
    \pscircle*(1.25,2){0.06}
    \pscircle*(1.75,2){0.06}
    \pscircle*(2.25,2){0.06}
    \pscircle*(2.75,2){0.06}
\psline(1.5,1)(0.25,2)
\psline(1.5,1)(0.75,2)
\psline(1.5,1)(1.25,2)
\psline(1.5,1)(1.75,2)
\psline(1.5,1)(2.25,2)
\psline(1.5,1)(2.75,2)
\psline(0.25,2)(1.25,2)
\rput(1.5,0.6){$v$}
\rput(0.25,2.3){$\gamma_1$}
\rput(0.75,2.3){$\gamma_2$}
\rput(1.25,2.3){$\gamma_3$}
\rput(1.75,2.3){$\gamma_4$}
\rput(2.25,2.3){$\gamma_5$}
\rput(2.75,2.3){$\gamma_6$}
\rput(1.5,0){$(a)$}
    \psline[showpoints=true](4.25,1)(4.75,1)(5.25,1)(5.75,1)(6.25,1)(6.75,1)
    \pscurve(4.25,1)(4.75,1.25)(5.25,1)
    \pscurve[linestyle=dashed](5.25,1)(5.75,1.25)(6.25,1.25)(6.75,1)
    \pscircle*(3.75,2){0.06}
    \psline[showpoints=true](4.25,2)(4.75,2)(5.25,2)(5.75,2)(6.25,2)(6.75,2)
    \pscircle*(7.25,2){0.06}
\psline(4.25,1)(3.75,2)
\psline(4.25,1)(4.25,2)
\psline(4.25,1)(4.75,2)
\psline(5.25,1)(4.75,2)
\psline(5.25,1)(5.25,2)
\psline(5.25,1)(5.75,2)
\psline(6.75,1)(5.75,2)
\psline(6.75,1)(6.25,2)
\psline(6.75,1)(6.75,2)
\psline(6.75,1)(7.25,2)
\rput(4.25,0.6){$v_1$} \rput(4.75,0.6){$v_2$}\rput(5.25,0.6){$v_3$}
\rput(5.75,0.6){$v_4$}\rput(6.25,0.6){$v_5$}\rput(6.75,0.6){$v_6$}
\rput(3.75,2.3){$\gamma_1$}
\rput(4.25,2.3){$\gamma_2$}
\rput(4.75,2.3){$\gamma_3$}
\rput(5.25,2.3){$\gamma_4$}
\rput(5.75,2.3){$\gamma_5$}
\rput(6.25,2.3){$\gamma_6$}
\rput(6.75,2.3){$\gamma_7$}
\rput(7.25,2.3){$\gamma_8$}
\rput(5.5,0){$(b)$}
 \end{pspicture}
\caption{book-1 Graph and Constraint-graph
}\label{fig-2}
\end{figure}

\begin{theorem}\label{proof-of-outerplanar-regularity}
Given a book-1 graph $G(V,E)$ and its outer-vertices $U$, $G$ has a
color assignment $cl$ using $\leq 3$ colors when it satisfies all
collection-constraints.
\end{theorem}
\begin{proof}
We make induction on $|V|$. When $|V|=1$, suppose $V=\{v\}$. By Observation~\ref{structure-of-constraint-graph},
the graph $G_{ct}\setminus G$ can be colored by $2$ colors, so $v$ can be colored by the third color.
Then by
Observation~\ref{trivial-satisfy-something}, the conclusion holds
when $|V|=1$.

When $|V|=n$, after removing a vertex, say $v_n$, the graph
$G'=G\setminus v_n$ is a book-1 graph and $U'=\{U\setminus v_n\}\cup
\{N(v_n)\}$ are outer-vertices. As graph $G_{ct}\setminus G$ is a $k_2$ subdivision, assume $\gamma$ is the cluster-vertex
at the end of the subdivision, whose cluster is noted as $c$, and the $v_n\in c$ is the vertex which is mostly closed to
the end of the spine according to the ordering along the spine of book-1 graph. Without losing generality, assume $G\setminus v_n$ is still connected, as the disconnected condition can be analyzed similarly.

In $G'$ clusters only on $\{v_n\}$
disappear and a new $cluster$ $\Upsilon_{N(v_n)}$ on $N(v_n)$
appears whose cluster-vertex $\gamma_{N(u)}=v_n$. Define
collection-constraints on $U'$ inheriting
from $U$. And if there is $cluster$ $\Upsilon_i$ with
$\{v_n\}\subset \Upsilon_i$, then set $cn(\Upsilon_{N(v_n)})\not=
cn(\Upsilon_i)$. By doing this, the constraint-graph of $G'$ is a
induced subgraph of $G_{ct}$.

By induction, $G'$ has a $\leq 3$ color assignment $cl'$ satisfying
all constraints on $U'$.
$cl'$ can be extended to be a color assignment $cl$ for $G$ by
setting $color_{cl}(v_n)=CL\setminus cn(\Upsilon_{N(v_n)})$. Assume
$cn(\Upsilon_{N(u)})=\{1,2\}$, then $color_{cl}(v_n)=\{3\}$. So
collections of clusters on $\{v_n\}$ can be $\{1,3\}$ or $\{2,3\}$.
Hence $\gamma$ vertices of clusters on $\{v_n\}$ can be colored with
$\{2,1\}$ respectively. 
Because of the relation between constraint-graphs of $G'$ and $G$,
$cl$ can be extended to be a $\leq 3$ color assignment of $G_{ct}$. By
Observation~\ref{trivial-satisfy-something}, next we need to show
that the disappearing cluster-vertices can get their cluster coloring.
Easy to see, coloring $\{1,2\}$ is enough to coloring all those cluster-vertices.
So we can conclude.

\qed
\end{proof}

\begin{corollary}\label{cor-regularity}
A book-1 graph has a color assignment using $\leq 3$ colors and outer-vertices are assigned $\leq 2$ colors.
\end{corollary}
\begin{proof}
We treat all the outer-vertices as a $cluster$. 
Then by Theorem~\ref{proof-of-outerplanar-regularity}, there is a
color assignment using $\leq 3$ colors and this $cluster$ is colored
with a collection of $2$ colors.
\qed
\end{proof}

There are close inner relations among outerplanar graphs, planar
graphs, and Hadwiger conjecture. Hadwiger conjecture when $k=4$ case
states that if a graph has no $K_4$ minor, its chromatic number is
$3$~\cite{Dirac1952}, so an outerplanar graph which has no $K_4$
minor can be colored by $3$ colors.

The $k=5$ case of hadwiger conjecture states that if a graph has no
$K_5$ minor, it is $4$-colorable. An alternative way to prove it is
to prove a planar graph which has no $K_5$ minor is $4$-colorable. We
show many coloring constraints can be satisfied for coloring an
outerplanar graph in this paper. In fact, such a methodology can be
used on planar graphs
which is promising leading to four color theorem or even arbitrary
$k$ of hadiwiger conjecture.

Follow-up work has been done.
In~\cite{WeiyaNote2}, a perimeter trace in planar graph is defined
similarly as for outerplanar graph and is proved to be boundary of
an outer face. Also a similar result as
Theorem~\ref{keep-structure-after-reducing} has been proved, hence
the induction method used in proof of
Theorem~\ref{proof-of-outerplanar-regularity} can be applied for
planar graph. More interesting, the extended subgraph in
Definition~\ref{extend-original-graph} is an union of path graphs ($k_2$ subdivision) in
outerplanar graph by Observation~\ref{structure-of-constraint-graph}, while it is an outerplanar graph in planar graph.
By expanding the consistency and inconsistency to $4$ colors, mapping perimeter trace in outerplanar graph to the unbounded face
in planar graph, and the $G_{ct}\setminus G$ from $k_2$ subdivision to the outerplanar graph, we can show that a planar graph can be colored
with $4$ colors with satisfying extra constraints. More details can be found at~\cite{WeiyaNote4}.

\section{Conclusion}\label{sec-conclusion}

In this paper, we defined perimeter trace of outerplanar graph, and
proved some properties of a perimeter trace. As it is known an
outerplanar graph is equivalent with a book thickness 1 graph, we
also discussed perimeter trace in a book-1 graph. Based on book-1
graph and with the help of perimeter trace, we prove that an
outerplanar graph has a special $3$-coloring assignment satisfying
special constraints. And such a proof does not depend on the
property that an outerplanar graph has its minimum degree $\leq 2$,
such results can be generalized to show similar
properties of planar graph. The similar definition of perimeter
trace on a planar graph has been fully discussed
in~\cite{WeiyaNote2}. Generalization of constraints on planar graph is discussed
at~\cite{WeiyaNote4}.


\begin{thebibliography}{99}
\bibitem{WeiyaNote1}
W. Yue, and W. Cao:
\newblock{\emph{An Equivalent Statement of Hadwiger Conjecture when $K=5$},}
\newblock{http://arXiv:1010.4321v5.}

\bibitem{WeiyaNote2}
W. Yue, and W. Cao:
\newblock{\emph{A New Proof of Wagner's Equivalence Theorem},}
\newblock{http://arXiv:1010.4321v5.}

\bibitem{WeiyaNote3}
W. Yue, and W. Cao: ``A Geometric View of Outerplanar Graph'',
http://arXiv:1010.4321v5.

\bibitem{HH1943}
    Hadwiger, Hugo, ``Uber eine Klassifikation der Streckenkomplexe", Vierteljschr. Naturforsch. Ges. Z¨¹rich 88: 133¨C143,
    1943.

\bibitem{Dirac1952}
    G. A. Dirac, `` Property of 4-Chromatic Graphs and some Remarks on Critical Graphs",
    Journal of the London Mathematical Society, Volume s1-27 Issue 1,
    pp. 85-92, 1952.
\end{thebibliography}

\begin{thebibliography}{99}
\bibitem{HH1943}
    Hadwiger, Hugo, ``\"{U}ber eine Klassifikation der Streckenkomplexe",
    Vierteljschr. Naturforsch. Ges. Z¨¹rich 88: 133-143,
    1943.

\bibitem{G1952}
    G. A. Dirac, `` Property of 4-Chromatic Graphs and some Remarks on Critical Graphs",
    Journal of the London Mathematical Society, Volume s1-27 Issue 1,
    pp. 85-92, 1952.

\bibitem{Wagner1937}
K. Wagner, ``\"{U}ber eine Eigenschaft der Ebenen Komplexe",
Mathematische Annalen, Volume 114(1), pp. 570-590, 1937.

\bibitem{Halin1964}
R. Halin, ``\"{U}ber einen Satz von K. Wagner zum Vierfarbenproblem",
Mathematische Annalen, Volume 153,
pp. 47-62, 1964.

\bibitem{Halin1967}
R. Halin, ``Zur Klassifikation der Endlichen Graphen nach H. Hadwiger und K. Wagner ",
Mathematische Annalen, Volume 172, pp. 46-78, 1967.

\bibitem{Ore1967}
O. Ore, ``The four-color problem'',  Pure and Applied Mathematics, Volume 27, Academic Press, New
York, 1967.

\bibitem{Young1971}
H. Peyton Young,
``A Quick Proof of Wagner's Equivalence Theorem'',
Journal of London Mathmatical Society (2), Volume 3, pp. 661-664, 1971.

\bibitem{Appel1977-1}
K. Appel, W. Haken, ``Every planar map is four colorable. Part I. Discharging'',
Illinois Journal of Mathematics, Volume 21, pp. 429-490, 1977.

\bibitem{Appel1977-2}
K. Appel, W. Haken and J. Koch, ``Every planar map is four
colorable. Part II. Reducibility'', Illinois Journal of Mathematics,
Volume 21, pp. 491-567, 1977.

\bibitem{Robertson1996}
   N. Robertson, D. P. Sanders, P. D. Seymour and R. Thomas,
   ``A new proof of the four colour theorem'',
   Electronic Research Announcements - American Mathematical Society, Volume 2, pp. 17-25, 1996.

\bibitem{Robertson1997}
   N. Robertson, D. P. Sanders, P. D. Seymour and R. Thomas,
   ``The four colour theorem'',
   Journal of Combinatorial Theory, Series B, Volume 70, pp. 2-44, 1997.

\bibitem{George2005}
G. Georges, ``A Computer-checked Proof of The Four Colour Theorem''.
\bibitem{George2008}
G. Georges, ``Formal Proof-The Four-Color Theorem", Notices of the
American Mathematical Society, Volume 55(11), pp. 1382-1393.

\bibitem{WeiyaNote2}
W. Yue, and W. Cao, ``An Alternative Proof of Wagner's Equivalence
Theorem", arXiv:1010.4321v5, submitted to The 24th Canadian
Conference on Computational Geometry.

\end{thebibliography}

\begin{thebibliography}{99}

\bibitem{CH1967}
G. Chartrand and F. Harary:
\newblock{\emph{Planar permutation graphs},}
\newblock{ Annales de l'institut Henri Poincar¨¦ (B) Probabilit¨¦s et Statistiques, Volume 3, No. 4, Pages 433-438, 1967.}


\bibitem{SM1979}
M. Syslo:
\newblock{\emph{Characterizations of outerplanar graphs},}
\newblock{Discrete Mathematics, Volume 26, Issue 1, Pages 47-53, 1979.}

\bibitem{DR2000}
R. Diestel:
\newblock{\emph{Graph Theory},}
\newblock{Graduate Texts in Mathematics, Springer-Verlag, Volume 173, 2000.}

\bibitem{BK1979}
Bernhart, Frank R. and Kainen, Paul C.:
\newblock{\emph {The book thickness of a graph},}
\newblock{Journal of Combinatorial Theory, Series B, Volume 27(3), pages 320-331, 1979.}

\bibitem{PS1986}
A. Proskurowski and M. Syslo:
\newblock{\emph{Efficient vertex-and edge-coloring of outerplanar graphs},}
\newblock{SIAM Journal on Algebraic and Discrete Methods, Volume 7, Issue 1, Pages 131-136, 1986.}

\bibitem{Dirac1952}
G. Dirac:
\newblock{\emph{Property of 4-Chromatic Graphs and some Remarks on Critical Graphs},}
\newblock{Journal of the London Mathematical Society, Volume s1-27, No 1, Pages 85-92, 1952.}

\bibitem{Wagner1937}
K. Wagner:
\newblock{\emph{\"{U}ber eine Eigenschaft der Ebenen Komplexe},}
\newblock{Mathematische Annalen, Volume 114, No 1, Pages 570-590, 1937.}

\bibitem{WeiyaNote1}
W. Yue, and W. Cao:
\newblock{\emph{An Equivalent Statement of Hadwiger Conjecture when $K=5$},}
\newblock{http://arXiv:1010.4321v5.}

\bibitem{WeiyaNote2}
W. Yue, and W. Cao:
\newblock{\emph{A New Proof of Wagner's Equivalence Theorem},}
\newblock{http://arXiv:1010.4321v5.}

\bibitem{WeiyaNote4}
W. Yue, and W. Cao:
\newblock{\emph{To Prove Four Color Theorem},}
\newblock{http://arXiv:1010.4321v5.}

\end{thebibliography}
\end{document}